\documentclass{article}[12pt]
\usepackage{a4}
\usepackage[english]{babel}
\usepackage[T1]{fontenc}
\usepackage[all]{xy}

\usepackage{amssymb}
\usepackage{amsmath}
\usepackage{amsthm}
\usepackage{multirow}
\usepackage{graphicx}
\usepackage{upgreek}
\usepackage{xcolor}
\usepackage{yfonts}

\newtheorem{thm}{Theorem}[section]
\newtheorem{prop}[thm]{Proposition}
\newtheorem{lem}[thm]{Lemma}
\newtheorem{coro}[thm]{Corollary}
\newtheorem{defn}[thm]{Definition}
\newtheorem{expl}[thm]{Example}
\newtheorem{rmk}[thm]{Remark} 

\newtheorem{conj}[thm]{Conjecture}
\newtheorem{nota}[thm]{Notation}
\numberwithin{equation}{section}

\newcommand{\lra}{\longrightarrow}

\newcommand{\mcalk}{\mathcal{K}}

\newcommand{\mcalo}{\mathcal{O}}

\newcommand{\mcalc}{\mathcal{C}}

\newcommand{\om}{\mathcal{O}(M)}

\newcommand{\mcala}{\mathcal{A}}
\newcommand{\mcalf}{\mathcal{F}}
\newcommand{\mcald}{\mathcal{D}}
\newcommand{\mcalt}{\mathcal{T}}
\newcommand{\mcalm}{\mathcal{M}}

\newcommand{\mcalb}{\mathcal{B}}

\newcommand{\mcall}{\mathcal{L}}
\newcommand{\mcalu}{\mathcal{U}}

\newcommand{\mcalr}{\mathcal{R}}

\newcommand{\tm}{\mathcal{T}^M}

\newcommand{\utm}{\mathcal{U}(\mathcal{T}^M)}
\newcommand{\butm}{\mathcal{B}_{\mathcal{U}(\mathcal{T}^M)}}

\newcommand{\phib}{\overline{\phi}}

\newcommand{\hra}{\hookrightarrow}
\newcommand{\lla}{\longleftarrow}

\newcommand{\ok}{\mathcal{O}_k}

\newcommand{\SsS}{\text{SsS}}

\newcommand{\mcale}{\mathcal{E}}

\newcommand{\ra}{\rightarrow}

\newcommand{\ahd}{\widehat{A}_{\bullet}}
\newcommand{\ah}{\widehat{A}}
\newcommand{\tmd}{\mathcal{T}_{\bullet}^{M}}

\newcommand{\dt}{\widetilde{\Delta}}
\newcommand{\dtd}{\widetilde{\Delta}^{\bullet}}
\newcommand{\dtn}{\widetilde{\Delta}^n}
\newcommand{\pai}{\partial^i}
\newcommand{\paj}{\partial^j}

\newcommand{\vzn}{\langle v_0, \cdots, v_n\rangle}
\newcommand{\vzin}{\langle v_0, \cdots, \widehat{v}_i, \cdots, v_n\rangle}

\newcommand{\fs}{f_{\sigma}}

\newcommand{\udo}{\mcalu(\Delta^1)}
\newcommand{\dn}{\Delta^n}
\newcommand{\ca}{\mcalc_A}
\newcommand{\fuca}{\mcalf(\utm; \ca)}
\newcommand{\mcalh}{\mathcal{H}}
\newcommand{\pdtn}{\partial \dtn}

\newcommand{\gb}{\overline{G}}
\newcommand{\pb}{\overline{p}}

\newcommand{\beb}{\overline{\beta}}
\newcommand{\mcalw}{\mathcal{W}}

\newcommand{\mcals}{\mathcal{S}}
\newcommand{\utmi}{\mcalu(\mcalt^{M_i})}
\newcommand{\utmii}{\mcalu(\mcalt^{M_{i-1}})}
\newcommand{\utmj}{\mcalu(\mcalt^{M_j})}
\newcommand{\utms}{\mcalu(\mcalt^{M_s})}
\newcommand{\utmr}{\mcalu(\mcalt^{M_r})}

\newcommand{\fib}{\overline{F}_i}
\newcommand{\fjb}{\overline{F}_j}
\newcommand{\nbb}{\mathbb{N}}
\newcommand{\ebb}{\mathbb{E}}
\newcommand{\la}{\leftarrow}
\newcommand{\fkam}{\mcalf_{kA} (\om; \mcalm)}
\newcommand{\foam}{\mcalf_{1A} (\mcalo(F_k(M)); \mcalm)}
\newcommand{\utfm}{\mcalu (\mcalt^{F_k(M)})}

\newcommand{\lnk}{\widetilde{\Lambda}^n_k}
\newcommand{\ahn}{\widehat{A}_n}

\newcommand{\pa}{\partial}

\newcommand{\fautm}{\mcalf_A(\utm; \mcalm)}
\newcommand{\futma}{\mcalf(\utm; \ca)}
\newcommand{\fabutm}{\mcalf_A(\butm; \mcalm)}
\newcommand{\mcalrb}{\overline{\mcalr}}
\newcommand{\lamb}{\overline{\Lambda}}
\newcommand{\ttb}{\overline{\Theta}}
\newcommand{\htah}{\underset{\text{sSet}}{\text{Hom}}(\tmd, \ahd)}
\newcommand{\shtah}{\underset{\SsS}{\text{Hom}}(\tm, \ahd)}

\newcommand{\shtoh}{\underset{\SsS}{\text{Hom}}(\tm, \ahd')}
\newcommand{\mcaln}{\mathcal{N}}

\newcommand{\wequ}{\simeq}

\newcommand{\haut}{\text{haut}}
\newcommand{\Nt}{\widetilde{N}}

\title{ \textbf{Classification of homogeneous functors in manifold calculus}}
\date{}

\author{Paul Arnaud Songhafouo Tsopm\'en\'e \\
Donald Stanley}

\begin{document}
\maketitle

\begin{abstract}
	For any object $A$ in a simplicial model category $\mathcal{M}$, we construct a topological space $\hat{A}$ which classifies homogeneous functors whose value on $k$ open balls is equivalent to $A$. This extends a classification result of Weiss for homogeneous functors into topological spaces.  
\end{abstract}

\tableofcontents


\section{Introduction}

Let $M$ be a manifold and let $\om$ be the poset of open subsets of $M$. In order to study the space of smooth embeddings of $M$ inside another manifold,  Goodwillie and Weiss \cite{good_weiss99, wei99} introduced the theory of manifold calculus, which is one incarnation of calculus of functors. One can define manifold calculus as the  study of contravariant functors from $\om$ to Top, the category of spaces. Being a calculus of functors, its philosophy is to take a  functor $F$  and replace it by its Taylor tower $\{T_k(F) \lra T_{k-1}(F)\}_{k \geq 1}$, which converges to the original functor in good cases, very much like  the approximation of a function by its Taylor series. The functor $T_kF$ is the \textit{polynomial approximation} to $F$ of degree $\leq k$. The \lq\lq difference\rq\rq{} between $T_kF$ and $T_{k-1}F$, or more precisely the homotopy fiber of the canonical map $T_kF \lra T_{k-1}F$,  belongs to a class of objects called \textit{homogeneous functors} of degree $k$. 
In \cite[Theorem 8.5]{wei99}, Weiss proves a deep result about the classification of  homogeneous functors of degree $k$. Specifically, he shows that any such functor is equivalent to a functor constructed out of a fibration $p \colon Z \to F_k(M)$ over the unordered configuration space of $k$ points in $M$, with a preferred section (germ) near the fat diagonal of $M$.

In this paper we classify homogeneous functors  of degree $k$ from $\om$ into any simplicial model category $\mcalm$. Such functors are determined by their values on disjoint unions of $k$ balls \cite[Lemma 6.5]{paul_don17-2}. Let $\mcalf_{kA} (\om; \mcalm)$ denote the category of {\it homogeneous functors} $F \colon \om \lra \mcalm$ (see Definition~\ref{defn:homogeneous_functor}) of degree $k$ such that $F(U) \simeq A$ for any $U$ diffeomorphic to the disjoint union of exactly $k$ open balls. 
Let $\mcalf_{kA} (\om; \mcalm)\slash \textgoth{w}$ denote the collection of weak equivalence classes of such functors. For spaces $X$ and $Y$, we let $[X, Y]$ be the standard notation for the set of homotopy classes of  maps from $X$ to $Y$.   We classify objects of $\mcalf_{kA} (\om; \mcalm)$ not through fibrations, but instead by maps from $F_k(M)$ to a certain topological space.  Specifically, we have the following, which is the main result of this paper.

\sloppy

\begin{thm}  \label{main_thm}
	Let $\mcalm$ be a simplicial model category, 
	and let $A \in \mcalm$. Then there is a topological space $\ah$ such that 
	for any  manifold $M$,
	\begin{enumerate}
		\item[(i)] if $k = 1$, there is a bijection
		\[
		\mcalf_{1A} (\om; \mcalm)\slash \textgoth{w} \ \  \cong \ \ \left[M, \ah\right].
		\]
		\item[(ii)] If $k \geq 2$ and $\mcalm$ has a zero object, there is a bijection
		\[
		\mcalf_{kA} (\om; \mcalm)\slash \textgoth{w} \ \  \cong \ \ \left[F_k(M), \ah\right]. 
		\]
	\end{enumerate}
\end{thm}

One may ask the following natural questions. 
\begin{enumerate}
	\item How is $\ah$ constructed? 
	\item What do we know about $\ah$? 
	\item How is our classification related to that of Weiss?
\end{enumerate}

To answer the first question,  let $\dtn, n \geq 0$, denote the poset whose objects are nonempty subsets of $\{0, \cdots, n\}$, and whose morphisms are inclusions. We construct $\ca \subseteq \mcalm$, a small subcategory consisting of a certain collection of fibrant-cofibrant  objects of $\mcalm$ that are weakly equivalent to $A$. Next we define $\ahd$ as the simplicial set whose $n$-simplices are contravariant functors $\dtn \lra \ca$ that take all maps to weak equivalences and are fibrant with respect to the injective model structure on $\mcalm^{\dtn}$.  Face maps are defined in the standard way, while degeneracies are more intricate (see Section~\ref {ahd_subsection}). Fibrancy of the diagrams will allow us to show that $\ahd$ is a Kan complex. We define $\ah$ as the geometric realization of $\ahd$.

For the second question, we do not know that much about  $\ah$. By definition it is connected, and we believe its  fundamental group is the  group  of (derived) homotopy automorphisms of $A$. Further computations seem hard. 

Regarding the third question, let us consider Weiss' result as mentioned above. In addition to this,  he proves  that the fiber $p^{-1}(S)$ of the fibration $p \colon Z \to F_k(M)$ that classifies a homogeneous functor $E \colon \om \to \text{Top}$ of degree $k$ is homotopy equivalent to $E(U_S)$, where $U_S$ is a tubular neighborhood of $S$ so that $U_S$ is diffeomorphic to a disjoint union of $k$ open balls \cite{wei99}. So the classification of objects of $\mcalf_{kA}(\om; \text{Top})$ amounts to the classification of fibrations over $F_k(M)$ with a section near the fat diagonal and whose fiber is $A$. In the case $k=1$, the fat diagonal is empty and we are just looking at fibrations over $M$ whose fiber is $A$. It is well known that there is a classifying space for such fibrations, namely $B\haut A$ where $\haut A$ denotes the topological/simplicial monoid of (derived) homotopy automorphisms of $A$. If $k > 1$ and $\mcalm = \text{Top}_*$, the category of pointed spaces, one has a similar classifying space for the objects of $\mcalf_{kA}(\om; \text{Top}_*)$. For a general simplicial model category $\mcalm$ we believe our classifying space, $\ah$,  is homotopy equivalent to $B\haut A$,  but we do not know how to prove this. We also believe there should be another approach (which does not involve $\ah$) to try to show that $\mcalf_{kA} (\om; \mcalm)\slash \textgoth{w} $ is in one-to-one correspondence with homotopy classes of maps $F_k(M) \to B\haut A$. We will say more about all this in Section~\ref{comparison_section}. 

One may use Theorem~\ref{main_thm} to set up the concept of \textit{characteristic classes} or \textit{invariants} of homogeneous functors though we do not know whether $\ah$ is homotopy equivalent to $B\haut A$. Let $F \in \fkam$ and let $f \colon F_k(M) \lra \ah$ denote the classifying map of $F$. One can define the \textit{characteristic classes} or \textit{invariants} of $F$ as the cohomology classes $f^*(H^*(\ah)) \subseteq H^*(F_k(M))$. If two functors are weakly equivalent, then by Theorem~\ref{main_thm} the corresponding maps are homotopic and therefore they are equal in cohomology. It would be interesting to see what kind of characteristic classes one could recover using our approach, or if other more traditional classifying spaces can be seen as special cases of our construction.

\textbf{Strategy of the proof of Theorem~\ref{main_thm}.} To prove the first part, we  need four intermediate results. 

Let $\tm$ be a triangulation of $M$, that is, a simplicial complex homeomorphic to $M$ together with a homeomorphism $\tm \lra M$. There is no need for the \textit{link condition} (which says that the link of any simplex is a piecewise-linear sphere). Associated with $\tm$ is the poset $\utm \subseteq \om$, which was introduced in \cite[Section 4.1]{paul_don17} and recalled in Definition~\ref{utm_defn}. That poset is one of the key objects of this paper as it enables us to connect many categories. 
Associated with $\utm$ is the poset $\butm \subseteq \om$ defined as 
\[
\butm = \{\text{$B$ diffeomorphic to an open ball such that $B \subseteq U_{\sigma}$ for some $\sigma$} \}.
\]
It turns out that $\butm$  is a basis for the topology of $M$. For a subposet $\mcals \subseteq \om$,  we denote by $\mcalf_A(\mcals; \mcalm)$ the category of isotopy functors $F \colon \mcals \lra \mcalm$ such that $F(U)$ is weakly equivalent to $A$ for any $U$ diffeomorphic to an open ball.  The first result we need is Lemma 6.5 from  \cite{paul_don17-2} which says that the categories $\mcalf_{1A} (\mcalo (M); \mcalm)$ and $\mcalf_A(\butm; \mcalm)$ are weakly equivalent (in the sense of \cite[Definition 6.3]{paul_don17-2}), that is, 
\begin{eqnarray} \label{res_eqn2}
\mcalf_{1A} (\mcalo (M); \mcalm) \wequ \mcalf_A(\butm; \mcalm). 
\end{eqnarray}
Looking closer at the proof of Lemma 6.5 from \cite{paul_don17-2}, one can see that the hypothesis that $\mcalm$ has a zero object is not needed when $k=1$.

Since  $\utm$ is a very good cover (in the sense of \cite[Definition 4.1]{paul_don17}) of $M$, we have the following which can be proved along the lines of \cite[Proposition 4.7]{paul_don17}.  
\begin{eqnarray}  \label{res_eqn3}
\mcalf_A(\butm; \mcalm) \wequ \fautm. 
\end{eqnarray}
(See also Proposition~\ref{fum_butm_prop}.) As we can see, (\ref{res_eqn2}) and (\ref{res_eqn3}) are just \lq\lq local versions\rq\rq{} of some results from \cite{paul_don17, paul_don17-2}. The following two technical results are new and proved using model category techniques. To state them, let $\ca$ as above. Let $\futma$ denote the category of isotopy functors from $\utm$ to $\ca$. 
By definition, there is an obvious functor  $\phi \colon \futma \hra \fautm$. Though we do not define a functor in the other direction, we succeed to prove that the localization of $\phi$ is an equivalence of categories. That is,
\begin{eqnarray} \label{res_eqn4}
\fautm [W^{-1}] \simeq  \futma [W^{-1}].
\end{eqnarray}
(See Proposition~\ref{fuca_prop}.) To get (\ref{res_eqn4}), we show that the localization of  $\phi$  is essentially surjective and fully faithful, the essentially surjectivity being the most difficult part.  The final result we need is stated as follows. Let $\ahd$ as above. One can associate to $\tm$ a 
canonical simplicial set denoted $\tmd$. We have the bijection  
\begin{eqnarray} \label{res_eqn5}
\futma \slash \textgoth{w} \  \  \cong  \  \ \left[\tmd, \ahd \right].
\end{eqnarray}
(See Proposition~\ref{fuca_iso_prop}.) To get (\ref{res_eqn5}) we construct explicit maps between the sets involved. The hardest part is to show that those maps are well defined. Defining $\ah$ as above, and noticing that the geometric realization of $\tmd$ is $M$, one deduces Theorem~\ref{main_thm}~-(i) from (\ref{res_eqn2})-(\ref{res_eqn5}). 

The second part of Theorem~\ref{main_thm} is an immediate consequence of the first part and the following weak equivalence, 
which is \cite[Theorem 1.3]{paul_don17-2}. 
\begin{eqnarray}   \label{res_eqn1}
\fkam \wequ \mcalf_{1A} (\mcalo (F_k(M)); \mcalm). 
\end{eqnarray}

Theorem~\ref{main_thm} has many hypotheses including the following:   $\mcalm$ is a \underline{simplicial} model category and \underline{$\mcalm$ has a zero object}. Note that the two underlined terms  are not needed to prove (\ref{res_eqn3})-(\ref{res_eqn5}).

\textbf{Outline}  \ The plan of the paper is as follows (see also the Table of Contents at the beginning
of the paper).  In Section~\ref{dtn_section}   we prove basic results we will use later. 
Section~\ref{ahd_section}  defines the 
simplicial set $\ahd$ and proves Proposition~\ref{ahd_prop}, 
which says that $\ahd$ is a Kan complex. 
First we construct a small category $\ca \subseteq \mcalm$ 
out of a model category $\mcalm$ and an object $A \in \mcalm$. 
Next we construct a specific fibrant replacement functor $\mcalr \colon \ca^{\dtn} \lra \ca^{\dtn}$,  which is essentially used to define degeneracy maps of $\ahd$. In Section~\ref{fautm_futma_section} we prove Proposition~\ref{fuca_prop} or (\ref{res_eqn4}). In Sections~\ref{lambda_section}, \ref{theta_section} we prove Proposition~\ref{fuca_iso_prop} or (\ref{res_eqn5}). Section~\ref{pmr_section} is dedicated to the proof of the main result of this paper: Theorem~\ref{main_thm}. Finally, in Section~\ref{comparison_section}, we state a conjecture saying how our classification is related to that of Weiss.

\textbf{Conventions and notation} \ These will be as in  \cite[Section 2]{paul_don17}, with the following additions. Throughout this paper the letter $M$ stands for a second-countable smooth manifold. The only place  we need  $M$ to be  second-countable is Section~\ref{infinite_case_subsection}.  We write $\mcalm$  for a  \emph{model category} \cite[Definition 1.1.4]{hovey99}, while $A$ is a fibrant-cofibrant object of $\mcalm$. As part of the definition, the factorizations in $\mcalm$ are functorial. Wherever necessary,  additional conditions on $\mcalm$ will be imposed. 
We write $[n]$ for the set $\{0, \cdots, n\}$,  $[n]_i$ for $[n] \backslash \{i\}$, and $[n]_{ij}$ for the set $[n] \backslash \{i, j\}$. Also we let $\{a_0, \cdots, \widehat{a}_i, \cdots, a_s\}$ denote the set $\{a_0, \cdots, a_s\} \backslash \{a_i\}$.  If $\mcals$ is a small category and $\mcalc$ is a subcategory of $\mcalm$, we write $\mcalc^{\mcals}$ for the category of contravariant functors  from $\mcals$ to $\mcalc$. An object of that category is called  \textit{$\mcals$-diagram}  or just \textit{diagram} in $\mcalc$. As usual,  weak equivalences in $\mcalc^{\mcals}$ are natural transformations which are objectwise weak equivalences.

A {\em weak equivalence diagram} is one where every morphisms is a weak equivalence.

\textbf{Acknowledgments} \ This work has been supported by  Pacific Institute for the Mathematical Sciences (PIMS), the University of Regina and the Natural Sciences and Engineering Research Council of Canada (NSERC), that the authors acknowledge. The second author worked on this paper while a visitor at  PIMS, the Fields Institute and Taida Institute for Mathematical Sciences (TIMS) and is grateful for their hospitality. 
We would like to thank Jim Davis for helping us with PL topology, and also the referee for their careful and thoughtful report which helped us improve the paper immensely.

\section{Basic results and recollections}  \label{dtn_section}

As we said in the introduction, this section presents some basic results we will use later. It is organized as follows.  In Section~\ref{dtn_subsection} 
we define the posets $\dtn$, $\pai \dtn$, and $\pdtn$, and endow the collection $\{\dt^n\}_{n \geq 0}$  with a natural  cosimplicial structure. 
In Section~\ref{model_category_subsection} we endow the category of $\dtn$-diagrams with the injective model structure and discuss some basic results about fibrant diagrams.
In Section~\ref{local_category_subsection}, we recall some classical facts about the localization of categories.

\subsection{Simplicial complexes, posets, simplicial and semi-simplicial sets}   \label{dtn_subsection}

Given $n\in \mathbb{N}\cup \{0\}$, let $[n]=\{ 0\leq i\leq n\}$ be the ordered set.

We consider two slightly different notions of simplicial complex, one including the empty set and one without. We will mostly use the one without the empty set and will make sure to remind you when we need to include the empty set. 
The corresponding posets of without includes into the one including the empty set.  the one including the empty set will only be used in Section \ref{ahd_section} in the construction of the degeneracies of the classifying space $A_n$.

A {\em simplicial complex} $K=(K,S)$ on the finite totally ordered set $S$ is a non-empty subset of the power set 
$\mathcal P(S)$ of $S$ closed under taking subsets. Often our ordered set $S$ will be a subset of $[n]$. The elements $\sigma\in K$ are called simplices, and $|\sigma |$ denotes their cardinality. For $\alpha\subset S$, anytime we write $\alpha = \{a_0, \cdots, a_s\}$, it will always mean  $a_0 \leq \cdots \leq a_s$.

A map between between simplicial complexes $f\colon (K,S)\rightarrow (K', S')$ is a map $f\colon S\rightarrow S'$ preserving the order (ie $\tau<\rho$ implies $f(\tau)<f(\rho)$), such that if $\sigma\in K$ then $f(\sigma)\in K'$. 

A subsimplicial complex of $K$ is a simplicial complex with a subset of the simplices of $K$ 

We can consider $\sigma\in K$ as a simplicial complex in its own right, as the smallest subsimplicial comlpex of $K$ containing $\sigma$, even though the meanings are slightly different we use the same notation for both. 

A simplicial complex $K$ is also considered a poset in the standard way, for $\tau,\sigma\in K$, $\tau\leq \sigma$ when $\tau\subset \sigma$. We also use $K$ to 
denote the associated poset category.

\begin{nota} \label{dtn_defn}
Let  $K$ be a simplicial complex. 
	For an object $\sigma \in K$, define $\partial\sigma$ as the simplicial complex whose simplices are nonempty proper faces of $\sigma$.	

	For $n \geq 0$,  we have the $n$-simplex $\dt^n={\mathcal P}([n])\setminus \{\emptyset\}$ and  its boundary $\partial\dtn=\mathcal P([n])\setminus \{[n],\emptyset \}$.  For a finite set $D$, 
	$\partial{\mathcal P}(D)={\mathcal P}(D)\setminus \{\emptyset \}$.
	
	 \end{nota}
	 
	 Some might consider the notation for $\partial{\mathcal P}(D)$ but we are only using it in Section \ref{ahd_section} with the opposite category of ${\mathcal P}(D)$. We could alternatively use complements, but that is less convenient. 
 
	\begin{enumerate}
		\item[(i)] For $i \in [n]$, define $\pai \dtn \subseteq \dtn$ to be the full subposet whose objects are nonempty subsets of $[n]_i=[n]\setminus \{i\}$. 
		\item[(ii)] Notice that 
		$
		\pdtn := \bigcup_{i=0}^n \pai \dtn.
		$ 
		\item[(iii)] 
 For $n \geq 1$ and   $0 \leq k \leq n$, define $\lnk \subseteq \dtn$ as the simplicial complex whose simplices are nonempty proper subsets of $[n]$ except $[n]_k$. That is,
	$
	\lnk = \pdtn \backslash [n]_k. 
	$

			\end{enumerate}

There is a unique isomorphism  between $\pai \dtn$ and $\dt^{n-1}$ preserving the order on the single element sets, $\pai \dtn \cong  \dt^{n-1}$. 
Similarly, one has $\pai \dtn \cap \paj \dtn \cong \dt^{n-2}$. We will often make these identifications throughout the paper without further comment.

\begin{rmk} \label{dtn_rmk}
	The poset $\dt^n$ has distinguished morphisms often called face maps, namely  
	\[
	d^i \colon \{a_0, \cdots, \widehat{a}_i, \cdots, a_s\} \lra \{a_0, \cdots, a_s\}, \quad 0 \leq i \leq s.
	\]
	One can check that every morphism of $\dt^n$ can be written as a composition of  $d^i$'s. 
	\end{rmk}

\begin{expl} \label{dt2_expl}
	The following diagram is the poset $\dt^2$. 
	\[
	\dt^2 = \xymatrix{   &   &   &   \{0\} \ar[ld]_-{d^1} \ar[rd]^-{d^1}    &    &  &      \\
		&   &  \{0, 1\}  \ar[rd]^-{d^2} &   &   \{0, 2\} \ar[ld]_-{d^1}  &   &    \\
		&  &     & \{0, 1, 2\}   &    &    &       \\	
		\{1\} \ar[rruu]^-{d^0} \ar[rrr]_-{d^1} &  &    &    \{1, 2\} \ar[u]^-{d^0}  &   &   & \{2\} \ar[lll]^-{d^0}  \ar[lluu]_-{d^0} }
	\]							
\end{expl}

Varying $n$ we get the collection $\dtd = \{\dt^n\}_{n \geq 0}$, which has a natural cosimplicial structure defined as follows. Let $\Delta$ be the category whose  objects  are partially ordered sets of the form $[n], n \geq 0$, with the natural order, and whose morphisms are non-decreasing maps. Let $d^i \colon [n] \lra [n+1]$ and $s^k \colon [n+1] \lra [n]$ be the special morphisms of $\Delta$ (see \cite[Section I.1]{goe_jar09}). It is well known that  $d^i$ and $s^k$ satisfy the cosimplicial identities.

\begin{defn} \label{di_sk_defn}
	\begin{enumerate}
		\item[(i)] Define a  functor $d^i \colon \dt^n \lra \dt^{n+1}, 0 \leq i \leq n+1$, as 
		$
		d^i(\{a_0, \cdots, a_s\}) := \{d^i(a_0), \cdots, d^i(a_s) \}.
		$
		\item[(ii)] Define a functor $s^k \colon \dt^{n+1} \lra \dt^n, 0 \leq k \leq n$, as 
		$
		s^k(\{a_0, \cdots, a_s\}) := \{s^k(a_0), \cdots, s^k(a_s)\}.
		$
	\end{enumerate}
	On morphisms, define $d^i$ and $s^k$ in the obvious way. 
\end{defn}

The following proposition is straightforward. 

\begin{prop} \label{cosimpl_rel_prop}
	The functors $d^i$ and $s^k$ we just defined satisfy the cosimplicial identities. That is, $\dtd = \{\dt^n\}_{n \geq 0}$ is a cosimplicial category. 
\end{prop}

\begin{nota}
We let $\text{sSet}$ denote the category of simplicial sets and $\text{SsS}$ denote the category of semi-simplicial sets. 
\end{nota}

Objects in $\text{sSet}$ can be thought of as contravariant functors $\Delta\rightarrow \text{Set}$ into the category of sets. Semi-simplicial sets are simplicial sets without the degeneracies. So if $\overline{\Delta}$ is the subcategory whose objects are those $\Delta$ but whose maps are only the inclusions (or face maps), a semi-simplicial set would then be a contravariant functor $\overline{\Delta}\rightarrow \text{Set}$.

Starting with a simplicial set we can forget the degeneracies to get a semi-simplicial set, this functor has an adjoint which gives a simplicial set who non-degenerate simplices are the simplices of the starting semi-simplicial set. 

We can think of a simplcial complex $K$ as a semi-simplicial set and we will then denote $K_{\bullet}$ as the associated simpicial set, we record the adjoint relationship for later use. The map on the left side of the equation corresponds to the restriction of the map on the right to non-degenerate simplices.

\begin{lem}\label{detbyrest}
For any simplicial complex $K$, and simplicial set $S_{\bullet}$, 
$${\underset{\SsS}{\text{Hom}}(K, S_{\bullet})}
\cong {\underset{\text{sSet}}{\text{Hom}}(K_{\bullet}, S_{\bullet})}.$$
\end{lem}

\subsection{Model categories}\label{modcat}

The paper uses a lot of model category  understanding of the results in the paper some familiarity with model categories is required. We use the definition from Hovey's book \cite[Definition 1.1.3]{hovey99}. In particular as part of its structure a model category has three specified classes of morphisms, cofibrations, fibrations and weak equivalences. The weak equivalences are the maps that become isomorphisms in the homotopy category. Also
in a model category
for each morphism $f \colon X \lra Y$ of $\mcalm$, we fix two functorial factorizations: 
\[\label{2fact}
\xymatrix{X \ \ \ar[rr]^-{f} \ar@{>->}[rd]_-{\sim}  &   & Y  \\
	&   V_f   \ar@{->>}[ru] &  } \qquad \xymatrix{X \ \ \ar[rr]^-{f} \ar@{>->}[rd]  &   & Y  \\
	&   W_f   \ar@{->>}[ru]_-{\sim} &  }
\]

The first factoring a map into a map that is both a weak equivlance and cofibriation followed by a fibration and the second factoring the map into a cofibration followed by a map that is both weak equivalence and a fibration. We are assuming both of these factorizations are functorial on maps between maps (in other words on commutative diagrams). Objects are called weakly equivalent if they can be connected by a sequence of weak equivalances (which can go in either direction). Weakly equivalent objects are the objects that become isomorphic in the homotopy category.  

Letting $\emptyset$ be a fixed initial object and $f\colon \emptyset \rightarrow X$ the unique map, then $W_f$ is denoted $QX$ and called the cofibrant replacement. Simlarly if $\ast$ denotes a fixed terminal object and $f\colon Y\rightarrow \ast$ the unique map, we get the fibrant replacement $RY=V_f$.

\subsection{Model category structure on $\mcalm^{\dtn}$}   \label{model_category_subsection}

In this section $\mcalm$ will denote a model category.

\begin{defn} \label{matching_defn}
	Let  $K$ be a simplicial complex. 
	\begin{enumerate}
		\item[(i)] Let $F \in \mcalm^{K}$. The \emph{matching object} of $F$ at $\sigma \in K$, denoted $M_{\sigma}(F)$, is defined as
		\[
		M_{\sigma} (F) := \underset{\alpha \in \partial\sigma}{\text{lim}} \; F(\alpha). 
		\]
		\item[(ii)] For $F \in \mcalm^K$ and $\sigma \in K$, the canonical map $
		F(\sigma) \lra M_{\sigma} (F),
		$
		provided by the universal property of limit, is  called the \emph{matching map} of $F$ at $\sigma$. 
	\end{enumerate}
\end{defn}

We record the following straightforward lemma for future use. 

\begin{lem}\label{same}
A cofunctor $F\colon\dt^n\rightarrow \mathcal C$ is the same as a cofunctor 
$F'\colon\partial \dt^n\rightarrow \mathcal C$ together with a map 
$F([n]) \rightarrow lim_{\partial\dt^n}F'$.  A functor $G\colon {\mathcal P}(D) \rightarrow {\mathcal C}$ is the same as a functor 
$G'\colon \partial {\mathcal P}(D) \rightarrow {\mathcal C}$ together with a map 
$colim_{\partial {\mathcal P}(D)}G'\rightarrow G(D)$.

\end{lem}

\begin{prop} \label{model_prop}
	Let $K$ be a simplicial complex. 
	There exists a model structure called the injective model category on the category $\mcalm^K$ of $K$-diagrams in $\mcalm$ such that weak equivalences and cofibrations are objectwise. 
	Furthermore, a map $F \lra G$ is a (trivial) fibration if and only if the induced map $F(\sigma) \lra G(\sigma) \times_{M_{\sigma}(G)} M_{\sigma} (F)$ is a (trivial) fibration for all $\sigma \in K$.

\end{prop}

\begin{proof}
	This follows from two facts. The first one is the fact that the poset $K$ is clearly a \textit{direct category} in the sense of \cite[Definition 5.1.1]{hovey99}. So any $F \in \mcalm^K$ is an \textit{inverse diagram} \footnote{An \textit{inverse diagram} is either a covariant functor out of a small inverse category \cite[Definition 5.1.1]{hovey99} or a contravariant functor out of a small direct category.} (remember that for us diagram means contravariant functor). The second fact is  \cite[Theorem 5.1.3]{hovey99}. 
\end{proof}

The category of diagrams, $\mcalm^K$, will mostly be always endowed with this model structure which is often called the injective model category structures. However when constructing the degeneracies of our classifying space  $\ahd$ we will use the projective model category structure whose fibrations and weak equivalences are pointwise (at the same time as having the empty set in our simplicial complexes). The dual result corresponding to Proposition~\ref{model_prop} also holds in the projective model category. 

\begin{coro}\label{fibtest}
	A diagram $F \colon K \lra \mcalm$ in the injective model category is fibrant if and only
 if all of its $k$-faces, $0 \leq k \leq n$, are fibrant $\dt^k$-diagrams. By \textit{$k$-face} of $F$, we mean 
a diagram $F\circ \tau$, where $\tau \colon \dt^k \hookrightarrow \dt^n$ is one of the canonical inclusion maps.  
Similarly for cofibrant diagrams in the projective model category structure. 
	
\end{coro}

\begin{prop} \label{fibrant_prop}
	Let $K$ be a finite simplicial complex. Let $F \colon K \lra \mcalm$ be a fibrant diagram.  Then the limit $\underset{\alpha \in K}{\lim} F(\alpha)$ is a fibrant object of $\mcalm$.
	
\end{prop}

\begin{proof}
	Since $F$ is fibrant, it follows that for every $\sigma \in K$, $F(\sigma)$ is fibrant. Furthermore the limit of the diagram $F$ is the same as its homotopy limit again because $F$ is fibrant.  Applying now \cite[Theorem 18.5.2]{hir03}, we get the desired result.  
\end{proof}

\begin{prop}  \label{we_prop}
	1) Let $F \colon \lnk \lra \mcalm$ be a fibrant weak equivalence diagram (ie one in which every morphism is a weak equivalence). Then for any $\alpha' \in \lnk$, the canonical projection $p_{\alpha'} \colon \underset{\alpha \in \lnk}{\lim} F(\alpha) \lra F(\alpha')$ is a weak equivalence. 
	
	2) Let $D$ be a finite set and $G\colon {\mathcal P}(D)^{\rm op} \lra\mcalm$ (or $\partial {\mathcal P}(D)^{\rm op} \lra\mcalm$) be cofibrant weak equivalence diagram. Then for any $\alpha' \in  {\mathcal P}(D)$ ($\partial {\mathcal P}(D)$), the canonical inclusion 
	$G(\alpha')\lra  \underset{\alpha \in  {\mathcal P}(D)}{colim} G(\alpha)$
	 ($ \underset{\alpha \in {\partial {\mathcal P}(D)}}{colim} G(\alpha)$)
	is a cofibration and a weak equivalence. 

\end{prop}

\begin{proof}
	By inspection, one can see that the indexing category $\lnk$ is contractible. One can then apply the dual of \cite[Corollary 1.18]{cis09} to get the first result. 
	
	Noting that ${\mathcal P}(D)$ has an initial object, the second result is similar. 
\end{proof}

\subsection{Localization of categories}   \label{local_category_subsection}

Again in this section $\mcalm$ is a model category.

The aim of this section is to recall some classical results  about the localization of categories we need. Our main references are \cite{dwyer_spa95} and \cite{hovey99}. The material of this section will be used in Section~\ref{fautm_futma_subsection}.

Let $\mcale$ be a  model category and let $\mcald \subseteq \mcale$ be a full subcategory (not necessarily a model subcategory) of $\mcale$. One should think of $\mcale$ as the diagram category $\mcalm^{\utm}$  and $\mcald$ as the category $\futma$ that will be introduced in Definition~\ref{fuca_defn}. We denote by $\mcalw_{\mcald}$ the class of weak equivalences of $\mcale$ that lie in $\mcald$. We will use the standard notation $\mcald[\mcalw_{\mcald}^{-1}]$ for the localization of $\mcald$ with respect to $\mcalw_{\mcald}$. Roughly speaking,  $\mcald[\mcalw_{\mcald}^{-1}]$ has the same objects as $\mcald$, and morphisms of $\mcald[\mcalw_{\mcald}^{-1}]$ are strings $(f_1, \cdots, f_r)$ of composable arrows where $f_i$ is either an arrow of $\mcald$ or the formal inverse $w^{-1}$ of an arrow $w$ of $\mcalw_{\mcald}$.

Recall the notion of cylinder object for $X \in \mcale$, denoted $X \times I$, from \cite[Definition 1.2.4]{hovey99}, and let $i_0$ and $i_1$ be the canonical maps from $X$ to $X \times I$. The following result is straightforward.

\begin{prop} \label{izio_prop}
	Assume that for any $X \in \mcald$ there is a  cylinder object $X \times I$ for $X$ in $\mcald$. Then one has $i_0 = i_1$ in the category $\mcald[\mcalw_{\mcald}^{-1}]$. 
\end{prop} 

For  $f, g \colon X \lra Y$ in $\mcald$, if $f$ is homotopic to $g$ (see \cite[Definition 1.2.4]{hovey99}), we will write $f \sim g$. 

\begin{prop}\cite[Corollary 1.2.6]{hovey99} \label{htpy_prop}
	Let $f, g \colon X \lra Y$ be morphisms of $\mcale$. Assume $X$ cofibrant and $Y$ fibrant.  If $f \sim g$ then there is a left homotopy $H \colon X \times I \lra Y$ from $f$ to $g$ using any cylinder object $X \times I$. 
\end{prop}

The following proposition follows by standard methods.

\begin{prop}  \label{ho_prop} Assume that $\mcald$ is closed under taking fibrant and cofibrant replacements and cylinder objects (ie each object in the category has a cylinder object in the category).
	Let $X, Y \in \mcald$. 
	\begin{enumerate}
		\item[(i)] Then there is a natural isomorphism 
		\[
		\xymatrix{\underset{\mcald}{\text{Hom}} (QX, RY)\slash {\sim} \;   \ar[rr]^-{\cong}_-{\theta'} &  & \underset{\mcald[\mcalw_{\mcald}^{-1}]}{\text{Hom}} (X, Y)}.
		\] 
		\item[(ii)] If $X$  and $Y$ are both fibrant and cofibrant, the map $\underset{\mcald}{\text{Hom}} \; (X, Y) \stackrel{\varphi}{\lra} \underset{\mcald[\mcalw_{\mcald}^{-1}]}{\text{Hom}} (X, Y)$ induced by the canonical functor $\mcald \lra \mcald[\mcalw_{\mcald}^{-1}]$ is surjective.
	\end{enumerate}
\end{prop}

\begin{proof}
	Part (i) follows along the lines of \cite[Theorem 1.2.10]{hovey99}. More precisely let $\mcald_{QR}$ be the subcategory of objects of the form $QRX$ for $X\in\mcald$. 
	Then since $\mcald$ is closed under fibrant and cofibrant replacement, we can follow the proof of \cite[Proposition 1.2.3]{hovey99} to see 
	that $\mcald_{QR}[\mcalw_{\mcald}^{-1}]\rightarrow \mcald[\mcalw_{\mcald}^{-1}]$ is an equivalence of categories, and since $\mcald$ is closed under taking 
	cylinder objects we can follow the proof of \break \cite[Corollary 1.2.9]{hovey99}
	to get an equivalence $\mcald_{QR}\slash{\sim}  \rightarrow \mcald_{QR}[\mcalw_{\mcald}^{-1}]$. The proof is completed using the natural isomorphism 
	 $\underset{\mcald}{\text{Hom}} (QRX, QRY)\slash{\sim} \rightarrow \underset{\mcald}{\text{Hom}} (QX, RY)\slash {\sim}$.

	Part (ii) comes from the fact that the following canonical triangle commutes and the map $\pi$ is surjective.
	\[
	\xymatrix{\underset{\mcald}{\text{Hom}} (X, Y)  \ar[rr]^-{\varphi}  \ar[rrd]_-{\pi} &  & \underset{\mcald[\mcalw_{\mcald}^{-1}]}{\text{Hom}} (X, Y)  \\
		&    &    \underset{\mcald}{\text{Hom}} (X, Y) \slash {\sim}. \ar[u]_-{\cong}^-{\theta'}}
	\]
\end{proof}

\section{The simplicial set $\ahd$}  \label{ahd_section}

In this section $\mcalm$ is a  model category, and $A$ is an object of $\mcalm$.  The goal here is to construct a simplicial set $\ahd$, which is a classifying space for linear functors whose value on a single ball is equivalent to $A$. We also show that $\ahd$ is a Kan complex (see Proposition~\ref{ahd_prop}). We need $\ahd$ to be a Kan complex  because of (\ref{res_eqn5}), which involves homotopy classes 
\footnote{If $X_{\bullet}$ and $Y_{\bullet}$ are two simplicial sets, one can consider the homotopy relation (see \cite[Section I.6]{goe_jar09}) in the set of simplicial maps from $X_{\bullet}$ to $Y_{\bullet}$. If one wants that relation to be an equivalence relation, one has to require $Y_{\bullet}$ to be a Kan complex \cite[Corollary 6.2]{goe_jar09}.} 
of simplicial maps into $\ahd$. This section is organized as follows. In Section~\ref{ca_subsection} we define a small category $\ca$, which will play the role of the target category for many functors including $n$-simplices of $\ahd$. In Section~\ref{mcalr_subsection} we define an explicit fibrant replacement functor $\mcalr \colon \ca^{\dtn} \lra \ca^{\dtn}$, which will be essentially used to define degeneracy maps of $\ahd$. Lastly, in Section~\ref{ahd_subsection}, we define $\ahd$ and prove Proposition~\ref{ahd_prop}.

\subsection{The category $\ca$}  \label{ca_subsection}

In Section~\ref{ahd_subsection}, we will define an $n$-simplex of $\sigma \in\ahd$ as a contravariant functor (sometimes called a diagram) from $\dtn$ to a very particular subcategory $\ca$ of $\mcalm$. This subcategory is closed under certain conditions and all of its objects are fibrant-cofibrant and of the same homotopy type. All of the maps in the diagram $\sigma$ must be weak equivalences. 

To guarantee that the collection of all $n$-simplices is actually a set, we need $\ca$ to be small. 

Recall the notation $V_f$ and $W_f$ from \ref{2fact}.

\subsubsection{Closure conditions}\label{clop}
Suppose $\mcalc\subset \mcalm$ is a full subcategory. 
Consider the following closure conditions on $\mcalc$. 

(A) Closed under cylinder objects. For any $X\in \mcalc$, $W_g\in \mcalc$ where $g$ is the fold map $ X\coprod  X \rightarrow X$. 

(B) Suppose we are given $X\in \mcalc$, functor $F\colon \partial \dtn\rightarrow \mcalc$ and map $\phi\colon X\rightarrow \underset{\alpha \in \partial \dtn}{\text{lim}}F(\alpha) $. If $F$ is fibrant in the injective model category on diagrams then $V_{\phi}\in \mcalc$. 

(C) Given 
$F\colon \partial \dtn\rightarrow \mcalc$, $G\colon \partial {\mathcal P}(D)^{\rm op}\rightarrow \mcalc$
and map $\phi\colon \underset{\alpha \in \partial {\mathcal P}(D)^{\rm op}}{\text{colim}}G(\alpha)\rightarrow \underset{\alpha \in \partial \dtn}{\text{lim}}F(\alpha) $. If $F$ is fibrant in the injective model category and $G$ is a weak equivalence diagram  and  cofibrant in the projective model category then $V_{\phi}\in \mcalc$. 

(D) Closed under fibrant and cofibrant replacement. If $X\in \mcalc$, then $QX,RX\in \mcalc$. 

(E) Suppose $F\colon \lnk\rightarrow \mcalc$ is a diagram where all the morphisms are weak equivalences. If $F$ is fibrant in the injective model structure then  $Q\lim_{\lnk} F \in \mcalc$. 


For any $\mcalc$ we let $\mcalc'$ denote the objects that can be created using (A)-(E), with elements of $\mcalc$ as inputs (so either as $X$ or as the values of the functors $F$ and $G$), so 
$\mcalc$ is closed 
under (A)-(E), is the same as saying that $\mcalc'\subset \mcalc$.

\begin{lem}\label{createsniceprops} Given any full subcategory $\mcalc\subset \mcalm$,

1)
There is a unique smallest (full) category $\overline{\mcalc} \subset \mcalm$ 
such that $\mcalc \subset \overline{\mcalc}$ and  $\overline{\mcalc}$ is closed under (A)-(E).

2) If $\mcalc$ is small then  $\overline{\mcalc}$ is small.

3) If all objects of $\mcalc$ are weakly equivalent and fibrant-cofibrant then all objects of 
  $\overline{\mcalc}$ are weakly equivalent and fibrant-cofibrant.

\end{lem}

\begin{proof}

This proof follows a small object argument (or compactness) type argument. 
Define a sequence of categories one for each natural number by

 $\mcalc(0)=\mcalc$, $\mcalc(i+1)=\mcalc(i)\cup \mcalc(i)'$. Then let $\overline{\mcalc}=\bigcup \mcalc(i)$. 
 
 For closure under (B), if we have a map $\phi$,  $X$ and the values 
 $F(\alpha)$ for $\alpha\in \partial \dtn$ all appear in some $C(i)$ since there are only finitely 
 many objects involved. Then $V_{\phi} \in C(i+1)$. The other cases are similar. 
 So we see that $\overline{\mcalc}$ is closed under (A)-(E), and clearly anything closed under (A)-(E) must contain 
 $\overline{\mcalc}$ proving uniqueness, so we have proved Part 1. Moving on to part 2,
  when $\mcalc$ is small so is $\mcalc'$ and so we see that $\overline{\mcalc}$ is small.

If $X$ is cofibrant and $Y$ is fibrant then $V_f$ and $W_f$ are fibrant-cofibrant. So we just need to check the objects that appear in the closure conditions for fibrancy or cofibrancy. 
 
   That all objects in $\mcalc$ 
 are weakly equivalent is straightforward in (A) and (B). For (C) it follows by Proposition \ref{we_prop}
 since $G$ is a weak equivalence cofibrant diagram that the colimit is cofibrant and weakly equivalent to $A$ and by Propositions \ref{fibrant_prop} that the limit is fibrant. This also shows fibrant-cofibrant in case (B). 
 
 Considering case (A) since $X$ is cofibrant, $X\coprod X$ is cofibrant, and of course $X$ is also fibrant. In 
 case (E) we use Propositions \ref{fibrant_prop} and \ref{we_prop} so see that the limit is fibrant and 
 weakly equivalent to $A$, taking cofibrant replacement gets us cofibrancy as well. 
 Case (B) can be done by the reader. 
 \end{proof}
 
 \begin{nota}\label{CAnotation}
For a fibrant-cofibrant object $A\in \mcalm$, we let $\ca$ denote $\overline{\{ A\}}$. So in particular by Lemma \ref{createsniceprops} $\ca$ is small, all of its objects are weakly equivalence and fibrant-cofibrant and it is closed under the closure conditions (A)-(E) above.
\end{nota}

\subsection{Special fibrant replacement functor $\ca^{\dtn} \lra \ca^{\dtn}$} \label{mcalr_subsection}

\begin{prop} \label{mcalr_prop}
Let $K\subset L$ be a subcomplex of a simplicial complex. 
	Suppose $F \colon L \lra \ca$ is a contravariant functor such that $F|_K$ is fibrant. Then there is a functor
	$\mcalr_{K,L}(F) \colon L \lra \ca$ such that:
	\begin{enumerate}
		\item[(i)] We have a natural transformation $\eta\colon F\rightarrow \mcalr_{K,L}(F)$
		\item[(ii)] For all $\sigma\in \dtn$, $\eta_{\sigma}$ is a cofibration, 
		a weak equivalence and the identity if $\sigma\in K$. This also implies that $\mcalr_{K,L}(F)$ is a weak equivalence diagram if $F$ is. 
		\item[(iii)] $\mcalr_{K,L}(F)\colon L\lra\ca$ is fibrant.
			\end{enumerate}
			 If $K\subset \dtn$ we sometimes write $\mcalr_{K,L}$ as 
		$\mcalr_K$. 
\end{prop}

\begin{proof}

	This follows from the standard contruction of the fibrant replacement in the injective model category structure. 
More precisely first consider the case when $K=\partial \dtn$ and $L=\dtn$, and let 
$\phi \colon F([n])\rightarrow \underset{\sigma \in \partial \dtn}{\lim} F(\sigma)$ 
be the induced map. Then letting 
$\mcalr_K(F)([n])=V_{\phi}$ and $\mcalr_K(F)(\sigma)=F(\sigma)$ otherwise, we complete the construction. Note that $V_{\phi}\in \mathcal \ca$ using the closure conditions (see Lemma \ref{createsniceprops} and Notation
\ref{CAnotation}). 

The general case comes from changing the functor one simplex $\sigma$ at a time. Note that outside of $\sigma$ the functor is not changed. Maps out of the new value at $\sigma$, $V_{\phi}$ factor through the limit and maps into $V_{\phi}$ are given by the composition through $F(\sigma)$.

\end{proof}

\subsection{The simplicial set $\ahd$}  \label{ahd_subsection}

To prove the main result of this paper (that is, Theorem~\ref{main_thm}), we construct a simplicial set 
 $\ahd$. In Proposition~\ref{ahd_prop} below we will prove that $\ahd$ is a Kan complex. Subsections \ref{sj_subsubsection} and \ref{tind} construct the degeneracies of $\ahd$. 
 
\begin{nota} \label{ahn_defn}
	Let $\ca$ be the category from Notation~\ref{CAnotation}, and let $\dt^n$ be the poset from Notation~\ref{dtn_defn}. 
		Define $\ah_n'$ as the collection of contravariant functors $\sigma \colon \dt^n \lra \ca$ such that 		
			 $\sigma$ is a weak equivalence diagram (ie each map in the diagram is a weak equivalence). Define $\ah_n\subset \ah_n'$ as the functors that in addition are 
			 fibrant $\dtn$-diagrams. 
			
\end{nota}
Note that for all $\alpha \in \dt^n$, $\sigma(\alpha)$ is weakly equivalent to $A$. Also note that each $\ahn'$ and $\ahn$ are sets since both categories $\dtn$ and $\ca$ are small (by Lemma~\ref{createsniceprops}).

 
There is a simplicial structure on $\ahd'$ induced by the cosimplicial structure of $\widetilde{\Delta}^{\bullet}$ (see Proposition~\ref{cosimpl_rel_prop}). However $\ahd'$ does not look like a Kan complex and in the proofs later on we need that our simplices are fibrant diagrams. To get these properties we use $\ahd$ even though this forces us to construct the degeneracies by hand.

\begin{defn}  \label{disj_defn}
	Recall the functor $d^i \colon \dt^{n-1} \lra \dt^n$  from Definition~\ref{di_sk_defn}. 
	\begin{enumerate}
		\item[(i)] Define the face map $d_i \colon \ah_n \lra \ah_{n-1}, 0 \leq i \leq n$,  as $d_i (\sigma) := \sigma d^i$. 
		\item[(ii)]  The degeneracy maps are defined in Section~\ref{sj_subsubsection} below.  
	\end{enumerate} 
\end{defn} 

Identifying along the isomorphism $\dt^{n-1}\cong \pai\dtn $ we can write $\sigma d^i=\sigma|_{\pai\dtn}$.

\subsubsection{Definition of the degeneracies on $\ah_{\bullet}$.}  \label{sj_subsubsection}

 The simplicial structure on $\ahn'$ comes from the standard cosimplicial 
structure on $\dt^{\bullet}$. We will write these using precomposition notation. So if $\sigma\in \ahn'$ we will write 
$\sigma d^i$ for a face map and $\sigma s^i$ for a degeneracy. 

Recall if $\sigma$ is also in $\ahn$, then we use the same face map and define as above 

\begin{equation}\label{compd}
d_i\sigma=\sigma d^i,
\end{equation}

these are just the restriction to a subsimplex and so $d_i$ 
preserves $\ah_{\bullet}$. Also the $d_i$ satisfy the simplicial identities.
We will define the degeneracies of $\ahn$ by recursion. We will actually define all compositions of degeneracies applied to non-degenerate simplices, this will automatically take care of the simplicial identities between them, and of course contains the same information as all of the maps $s_i\colon \ahn\rightarrow \ah_{n+1}$.

We can consider (compositions of) degeneracies $s^K\colon \dt^l\rightarrow \dt^t$ as order preserving surjective maps between the sets $[l]$ and $[t]$. These are the compositions of single identifications between adjacent 
elements of the sets, written as $s^i$ in Definition \ref{di_sk_defn}, and the set of identifications used in the degeneracy determines the degeneracy $s^K$. Write the set of identifications in $s^K$ as $D$. 
In decompositions of $s^K=s^Is^J$ as $\dt^l\stackrel{s^J}{\rightarrow}\dt^s \stackrel{s^I}{\rightarrow}\dt^t$ 
(or $[l]\stackrel{s^J}{\rightarrow}[s] \stackrel{s^I}{\rightarrow}[t]$), since $s^J$ is surjective, $s^J$ is determines $s^I$. The possible  decompositions are determined by $s^J$ and are in bijection with 
 ${\mathcal P} (D)$. Slightly abusing notation we will ignore the bijection and write $s^J\in {\mathcal P} (D)$. Put the subset order on ${\mathcal P} (D)$. Observe that if $s^J\leq s^{J'}$ if and only if you can decompose $s^{J'}$ as $s^Is^J=s^{J'}$.  Note that decompositions, of $s^K$ are compatible with composition with other degeneracies, for example if $s^K=s^Is^J$ then $s^is^K=(s^is^I)s^J$ and 
 $s^Ks^i=s^I(s^Js^i)$.  So we get $D\rightarrow D'$ where $D'$ is the set of identifications in $s^is^K$ or $s^Ks^i$.  
 
 We have to be careful how we compose the two different face and degeneracy maps. 
It helps to remember that the $d^i$ and $s^i$ can be applied to all of the bigger simplicial set $\ahn'$. 
So we understand the $d_i$ and $s_i$ as applying first and keeping inside of $\ahn$ and the $d^i$ 
and $s^i$ as applying after. So $d_is_I\sigma s^J$ should be read $(d_is_I \sigma)s^J$ and would 
also be equal to $s_I\sigma d^is^J$. As is usually when turning right modules to left modules, the lower indices compose in the opposite direction. So for example if $s^J=s^is^j$ then $s_J=s_j s_i$. 

\subsubsection{The induction}\label{tind}

\begin{prop}
There is a simplicial structure on $\ah_{\bullet}$ whose face maps $d_i$ are the same as those in 
$\ah_{\bullet}'$.

\end{prop}

\begin{proof}
 
Recall in the decomposition of $s^K$,  $s^Is^J$, $s^J$ also determines the other map $s^I$, 

Assume for every $l<m$, non-degenerate $\sigma \in \ah_{|\sigma|-1}$, degeneracy 
$s^K\colon, \dt^l\rightarrow \dt^{|\sigma|-1}$ and decomposition $s^J\in {\mathcal P} (D)$, 
the following:

We have defined $\eta(s^J, \sigma)=s_{I}\sigma s^J\in  \ah'_{l}$ and natural transformations 
$\eta((s^J,s^{J'}),\sigma)\colon s_I\sigma s^J \rightarrow s_{I'}\sigma s^{J'}$ when $s^{J'}\subset s^J$ such that 
$\eta((s^{J'},s^{J''}),\sigma)\eta((s^J,s^{J'}),\sigma)=\eta((s^J,s^{J''}),\sigma)$ when  $s^{J''}\subset s^{J'}\subset s^J$. 

In other words we have a functor 
$\eta(\sigma)=\eta(\_, \sigma)\colon {\mathcal P}(D)^{\rm op} \rightarrow \ca^{\dt^l}$

If $s^J=\emptyset$ then we write $s_K\sigma=s_K\sigma s^{\emptyset}$. We assume that 
$s_K\sigma$ is fibrant. In other words that $s_K\sigma\in \ah_l$.

We make a few more assumptions:

1) $\eta(\sigma)$ is acyclic. More precisely  
$\eta((s^J, s^{J'}), \sigma)(\alpha)$ is a weak equivalence for every $\alpha \subset [l]$ (ie for every simplex in 
$ \dt^l$). 

2) $\eta(\sigma)$ is acyclic cofibrant in the top dimension. More precisely 
$\eta(\_, \sigma)([l])\colon {\mathcal P} (D)^{\rm op}\rightarrow \ca$ is a weak equivalence diagram that is cofibrant in the projective model category structure. 

By Corollary \ref{fibtest} this is equivalent to saying that for every $s^J\in  {\mathcal P} (D)$,
 $colim_{s^J\subsetneq s^I} \eta(s^I, \sigma) ([l])\rightarrow \eta(s^J,\sigma)([l])$ is an acyclic cofibration. 

3) The maps $\eta((s^J,s^{J'}),\sigma)$ are compatible with degeneracies. This means, 

$$\eta((s^J,s^{J'}),\sigma)s^i=\eta((s^Js^i,J's^i),\sigma ).$$

Unwinding we see this is an equality of maps $s_I\sigma s^J s^i \rightarrow s_{I'}\sigma s^{J'} s^i$ between functors. 

4) The maps $\eta((s^J,s^{J'}),\sigma)$ are compatible with the face 
maps $d^i$.

For face maps two things can happen. The $d^i$ can cancel one of the degeneracies in which case $s^Kd^i=s^{K'}$ and this means we get formulas of the form 

$$
\eta((s^J,s^{J'}),\sigma)d^i=\eta((s^M,s^{M'}),\sigma).
$$
Note there $s^M$ and $s^{M'}$ are the first part of the decompositions of $s^{K'}$. 
On the other hand if $d^i$ can commute with $s^K$ we have $s^Jd^i=d^ks^M$ and $s^Nd^k=d^js^L$
and   
we can write $d_j\sigma=s_N\tau$ with $\tau$ non-degenerate. 
Similarly for $J'$ we get $s^{I'}s^{J'}d^i=d^js^{L'}s^{N'}$. 
Then we have a formula

$$
\eta(s^J,s^{J'}),\sigma)d^i=
\eta((s^Ls^N,s^M),(s^{L'}s^N,s^{M'}),\tau)
$$

To prove the induction step 
for every sequence of degeneracies $s_K$ and non-degenerate simplex 
$\sigma\in \ahn$ such that $|K|+n=m$ we need to:

A) Construct $s_K\sigma$, and observe it satisfies the simplicial identities of the form $d_is_K\sigma={\rm something}$
and $s_K\sigma=s_Is_J \sigma$. 

B) For every decomposition $s^K=s^Is^J$ construct $s_I\sigma s^J\in \ah_m'$, satisfying 3) and 4) above. 

C) For every two decompositions $s_K=s_Is_J$ and $s_K=s_{I'}s_{J'}$ with $J\leq J'$ 
construct $\eta((s^J,s^{J'}),\sigma)$ satisfying 1) and 2) above. 

To start off observe that if $s^J\not=\emptyset$ in B) or  $s^{J'}\not=\emptyset$ in C) the objects  
$s_I\sigma s^J\in \ahn'$ and maps $\eta((s^J,s^{J'}),\sigma)$ are determined by 3) above. 

Next Equation \ref{compd} determines what $s_K\sigma|_{\partial^i \dtn}=s_K\sigma d^i$ must be and so we define them as such. Also 2) implies that we have maps 
$s_I\sigma s^J d^i\rightarrow s_K\sigma d^i$ which are compatible (with respect to the 
faces of $\partial \dtn$) and natural transformations 
$s_I\sigma s^J|_{\partial \dtn}\rightarrow s_K\sigma|_{\partial \dtn}$ which are also compatible with the 
maps  $\eta((s^J,s^{J'}),\sigma)$ when $s^{J'}\not=\emptyset$. 

Using Lemma \ref{same} this gives us  the lower triangle in the commutative diagram

\begin{equation}\label{gott}
\xymatrix{
\underset{\emptyset\not=J\in {\mathcal P}(D)^{\rm op}}{\text{colim}} s_I\sigma s^J([m]) \ar[drr] \ar[d] 
\ar@{-->}^{\phi}[rr] && s_K\sigma([m]) \ar@{-->}[d] \\
\underset{\emptyset\not=J\in {\mathcal P}(D)^{\rm op}}{\text{colim}}  \underset{\partial \dtn}{\lim} s_I\sigma s^J \ar[rr] 
&& \underset{\partial \dtn}{\lim} s_K\sigma 
}
\end{equation}

We get the dashed part of the diagram by factoring the diagonal map into an 
acyclic cofibration followed by a fibration. Again using Lemma \ref{same} and the closure conditions this defines $s_K\sigma\in \ah_m$, and by including into the colimits all 
maps $\eta((s^J,s^\emptyset),\sigma)$. 

Note that condition 1) says $\sigma s^K([m])\rightarrow s_K\sigma ([m])$ is an equivalence, so since for any $\tau\subset [m]$, the maps 
$\sigma s^K([m])\rightarrow \sigma s^K(\tau)$ and  
$\sigma s^K(\tau)\rightarrow s_K\sigma (\tau)$ are equivalences so is 
$s_K\sigma ([m])\rightarrow s_K\sigma (\tau)$. Thus all the maps in the diagram
$s_K\sigma$ are equivalences. 

Condition 1) is satisfied since the constructed map $\phi$ is an acyclic cofibration. By Corollary \ref{fibtest}
Condition 2) can be checked on simplices of ${\mathcal P}(D)$ one at a time. On the simplex $D$ it is the 
construction of in Diagram \ref{gott} and on other simplices it  condition 2) for lower dimensions together with 3). 

For Condition 3) observe that guaranteeing compatibility with degeneracies was what was used to construct the functor 
$\eta(\sigma)$ restricted to $\partial {\mathcal P}(D)$. Similarly guaranteeing compatibilities of the form
$d_is_K={\rm something}$ and Condition 4) 
was used to construct the functor restricted to $\partial \dtn$. Note all that all the somethings are of one degree lower and so are already defined by the induction hypothesis. The simplicial identities with only $d_i$ follow since $s_K\sigma$ are diagrams (they are elements of $\ahn'$). The simplicial identities $s_K\sigma=s_Is_J\sigma$ are baked into the construction since for any such decomposition the right hand side is defined by the right hand side. 

This completes the induction and the construction of the degeneracies making $\ah_{\bullet} $ into a simplicial set. 
\end{proof}

Notice that the construction of degeneracies in the above proposition is natural. We will not make use of this.

\subsection{Proving that $\ahd$ is a Kan complex}

\begin{prop}  \label{ahd_prop}
	The simplicial set $\ahd$  is a Kan complex.
\end{prop}

\begin{proof}
	Let $n \geq 0$, and let $k \in \{0, \cdots, n\}$. Consider $n$ $(n-1)$-simplices $\sigma_0, \cdots, \widehat{\sigma}_k, \cdots,  \sigma_n$ of $\ahd$ such that 
	$
	d_i (\sigma_j) = d_{j-1} (\sigma_i), i < j,$  and $i, j \neq k.
	$
	Our goal is to construct an $n$-simplex $\sigma \in \ah_n$ such that $d_i\sigma = \sigma_i$ for all $i \neq k$. We first need to construct an intermediate functor $\overline{\sigma} \colon \dtn \lra \ca$. 
	
	Recall the 
	poset $\lnk$ from Notation~\ref{dtn_defn}.  For $i \in [n], i \neq k,$ define  $\overline{\sigma} \colon \lnk \lra \ca$ as $\overline{\sigma} (\alpha) := \sigma_i (\alpha)$ for $\alpha \in \pa^i \dtn \cong \dt^{n-1}$.  	
	It is straightforward to see that this is well defined on the intersection $\pai \dtn \cap \paj \dtn$. Note this is fibrant by Corollary \ref{fibtest} since each $\sigma_i$ is. 
	
	Now define  
	\[
	\overline{\sigma}([n]) :=Q \underset{\alpha \in \lnk}{\lim} \overline{\sigma}(\alpha), \quad \overline{\sigma}([n]_k):= \overline{\sigma}([n]), \quad \text{and} \quad \overline{\sigma}(d^k):= id,
	\]
	where $d^k \colon [n]_k \lra [n]$ is a morphism of $\dtn$. For $\alpha \in \lnk$, and 
	$d^{\alpha[n]} \colon \alpha \lra [n]$ a morphism of $\dtn$, define $\overline{\sigma}(d^{\alpha[n]})$ as  the composition 
	$$
	Q \underset{\alpha \in \lnk}{\lim} \overline{\sigma}(\alpha)\rightarrow  \underset{\alpha \in \lnk}{\lim} \overline{\sigma}(\alpha) \stackrel{p_{\alpha}}{\rightarrow} \overline{\sigma}(\alpha)
	$$. 
	Where $p_{\alpha}$ is the projection. 
	By Proposition \ref{we_prop} $\overline{\sigma}$ is a weak equivalence diagram and by Lemma
	\ref{createsniceprops} (closure condition (E)) and Notation \ref{CAnotation}
	all its values are in $\ca$ so we get $\overline{\sigma} \colon \dtn \lra \ca$ such
	that $\overline \sigma|_{\lnk}$ is fibrant. Then by Proposition \ref{mcalr_prop}
 $\sigma=\mcalr_{\lnk}(\overline \sigma)\colon \dtn \lra \ca$ is fibrant with 
 $ \sigma|_{\lnk}=\overline \sigma|_{\lnk}$ and all maps still weak equivalences. Hence 
	$\sigma\in \ahn$ and $d_i \sigma = \sigma_i$ for all $i \neq k$. This proves the proposition.  
\end{proof}

\section{The functor categories $\fautm$ and $\futma$ }   \label{fautm_futma_section}

From now on, we let $\tm$ denote a triangulation of $M$.  Associated to the triangulation we get a simplicial complex consisting of the (in this case geometric) simplices of the triangulation, which we denote by $P(\tm)$. As per our convention this also also denotes the associated poset (and its category). We find it helpful to separate out $P(\tm)$ from the more topological $\tm$, even though they have essentially the same set of simplices. 

Let $\ca \subseteq \mcalm$  be  the small subcategory constructed in  Section~\ref{ca_subsection}. In 
this section we recall an important poset $\utm$ and introduce two categories: $\fautm$ and $\futma$. The first one is the category of isotopy functors $\utm \lra \mcalm$ that send every object to an object weakly equivalent to $A$, while the second is the category of isotopy  functors from $\utm$ to $\ca$. By the definitions there is an inclusion functor $\phi  \colon \futma \hra \fautm$. The goal of this section is to prove (\ref{res_eqn4}) or Proposition~\ref{fuca_prop} below, which says that the localization of $\phi$ with respect to weak equivalences is an equivalence of categories. We begin with the definition of $\utm$.

\subsection{The  poset $\utm$}

The poset $\utm$ was introduced by the authors in \cite[Section 4.1]{paul_don17}. We recall its definition, and refer the reader to \cite[Section 4.1]{paul_don17} for more explanation. For the definition of barycentric subdivision see \cite[Chapter 2, page 119]{hatcher02}. 

First take two barycentric subdivisions $BB\tm$ of $\tm$, next for a simplex $\sigma$ of $\tm$ look at the star neighborhood of $\sigma$ in $BB\tm$. This consists of all simplices that share a vertex with $\sigma$. Since we have subdivided $\sigma$ in $BB\tm$ consists of many simplices. 

Note:

1) The star neighborhood of $\sigma$ is the same as the union of the star neighborhoods of all the vertices in the subdivided $\sigma$. 

2) Because we have subdived twice if we take adjacent vertices in $\tm$, then their stars neighborhoods in $BB\tm$ do not intersect, and subdividing twice is needed since if we only subdivide once star neighborhoods of adjacent vertices will intersect. 

So we define a category:

\begin{defn} \label{utm_defn}
For a simplex $\sigma \in \tm$ define $U_{\sigma}$ as the interior of the star of $\sigma$ in $BB\tm$. Putting these into a category $\utm$, 
	an object of $\utm$ is defined to be $U_{\sigma}, \sigma \in \tm$. There is a morphism $U_{\sigma} \lra U_{\sigma'}$ if and only if $\sigma$ is a face of $\sigma'$. In other words, morphisms of $\utm$ are just inclusions. 
\end{defn} 

\begin{rmk} \label{utm_pt_rmk}
A geometric $n$-simplex is denoted $\dn$. This can be thought of as a subspace of $\mathbb{R}^{n+1}$. These are used when building the triangulation $\tm$ of our manifold $M$. Once we consider the simplicial complex associated to $\tm$ and 
its poset $P(\tm)$, we get that one has a canonical isomorphism  $\utm \stackrel{\cong}{\lra} P(\tm), U_{\sigma} \mapsto \sigma$. 
	When $\tm = \dn$, there is an isomorphism $\mcalu(\dn) \cong \dtn$, where $\dtn$ is the poset from Definition~\ref{dtn_defn}. 
\end{rmk} 

\begin{rmk} \label{utm_rmk}
	\begin{enumerate}
		\item[(i)] It is clear that every morphism of $\utm$ is a composition of $d^i$'s, where 
		\begin{eqnarray} \label{di_eqn}
		d^i \colon U_{\langle v_{a_0}, \cdots, \widehat{v}_{a_i}, \cdots, v_{a_n}\rangle} \hra U_{\langle v_{a_0}, \cdots, v_{a_n}\rangle}
		\end{eqnarray}
		\item[(ii)] It is also clear that every morphism of $\utm$ is an isotopy equivalence since the inclusion of one open ball of $M$ inside another one  is always an isotopy equivalence \cite[Chapter 8]{hirsch76}. 
	\end{enumerate} 
\end{rmk}

\subsection{The functor categories $\fautm$ and $\fabutm$}\label{2fun}

Here we prove (\ref{res_eqn3}) or Proposition~\ref{fum_butm_prop} below. We begin with a few definitions.

\begin{defn} \label{butm_defn}
	 Define $\butm$ to be the collection of all subsets $B$ of $M$ diffeomorphic to an open ball such that $B$ is contained in some $U_{\sigma} \in \utm$. 
\end{defn}

Observe $\butm$ is a basis for the topology of $M$ that contains $\utm$. Often $\butm$ will be thought of as 
the poset whose objects are $B \in \butm$, and whose morphisms are inclusions. 

We now define what is called an isotopy functor. First, recall that $\om$ is the category whose objects are open subsets of $M$ and morphisms are inclusions.  Recall Section \ref{modcat} for the terms weak equivalence and weakly equivalent.

\begin{defn} If $\mcals$ is a subcategory of $\om$, a functor $F \colon \mcals \lra \mcalm$ is an {\em isotopy functor} if it is objectwise fibrant and if it sends isotopy equivalences to weak equivalences.  
\end{defn}

Note that if the subcategory $\mcals$ consists only of balls, then it follows from Remark \ref{utm_rmk} that isotopy functors are the same as weak equivalence diagrams.

\begin{defn} \label{fum_defn}
	Let $\mcals \subseteq \om$ be a subcategory. 
	Define $\mcalf_A(\mcals; \mcalm)$ to be the category whose objects are isotopy functors 
	$F \colon \mcals \lra \mcalm$ such that for every $U \in \mcals$, $F(U)$ is weakly equivalent to $A$. As usual morphisms are natural transformations. 
\end{defn}

We are mostly interested in the case when $\mcals = \utm$ or  $\mcals=\butm$.

\begin{defn} \label{defn:we_categories}
	Let $\mcalc$ and $\mcald$ be categories both equipped with a class of morphisms called weak equivalences. 
	 \begin{enumerate}
	 	\item[(i)] We say that two functors  $F, G \colon \mcalc \lra \mcald$ are  \emph{weakly equivalent} 
	 	if they are connected by a zigzag of objectwise weak equivalences. 
	 	\item[(ii)] A functor $F \colon \mcalc \lra \mcald$ is said to be a \emph{weak equivalence} if it satisfies the following two conditions.
	 	\begin{enumerate}
	 		\item[(a)] $F$ preserves weak equivalences.
	 		\item[(b)] There is a functor $G \colon \mcald \lra \mcalc$ such that $FG$ and $GF$ are both weakly equivalent to the identity. The functor $G$ is also required to preserve weak equivalences. 
	 	\end{enumerate} 
	 	\item[(iii)] We say that $\mcalc$ is weakly equivalent  to $\mcald$, and we denote $\mcalc \simeq \mcald$, if there exists a zigzag of weak equivalences between $\mcalc$ and $\mcald$. 
	 \end{enumerate} 
\end{defn}
 

\begin{prop} \label{fum_butm_prop}
	The category $\fautm$ is weakly equivalent 
	 to the category $\fabutm$. 
	That is,
	\[
	\fabutm \simeq \fautm. 
	\]
\end{prop}

\begin{proof}	
 Define $\varphi \colon \fabutm \lra \fautm$ as the restriction to $\utm$. That is,  $\varphi(G):= G|\utm$.  To define a functor $\psi$ in the other way, let $F \colon \utm \lra \mcalm$ be an object of $\fautm$. 
 Take any $B\in \butm$. It is fully inside one of the $U_\sigma$ and we can take the smallest simplex $\sigma_B$ such that $B\in U_{\sigma_B}$ define $\psi(F)(B):= F(U_{\sigma_B})$.  If $\beta \colon F \lra F'$ is a morphism of $\fautm$, define $\psi(\eta)(B):= \beta[U_{\sigma_B}]$. It is straightforward to see that $\varphi$ and $\psi$ preserve weak equivalences. It is also straightforward to check that $\phi \psi = id$ and $\psi \phi \simeq id$. This proves the proposition.  
\end{proof}

\subsection{Proving that $\fautm  \ \simeq \ \futma$}   \label{fautm_futma_subsection}

In this section we again use the injective model category structure on $\mcalm^{\utm}$.
The goal here is to prove (\ref{res_eqn4})  or Proposition~\ref{fuca_prop} below, which says that the localization of $\fautm$ is equivalent  to the localization of a certain small category $\futma$ that we now define.

\sloppy

\begin{defn} \label{fuca_defn} 
	\begin{enumerate}
		\item [(i)] Define $\futma$ as the category of isotopy functors from $\utm$ to $\ca$. 
		To simplify the notation, we will often write $\mcalk$ for $\fuca$ in this section. That is, $\mcalk := \fuca.$ Weak equivalences of $\mcalk$ are natural transformations which are objectwise weak equivalences. We denote the class of weak equivalences of $\mcalk$ by $\mcalw_{\mcalk}$. 
		\item[(ii)] Define $\mcall := \fautm$ (see Definition~\ref{fum_defn}). We denote the class of weak equivalences  of $\mcall$ by $\mcalw_{\mcall}$ . 
	\end{enumerate}
\end{defn}

By Definition~\ref{fuca_defn} and since $\ca$ is a subcategory of $\mcalm$, one has $\mcalk \subseteq \mcall$. 

\begin{rmk} Recall we use the injective model category structure on $\mcalm^{\utm}$. 
	The subcategory $\mcalk \subseteq \mcalm^{\utm}$ might not be a model category in its own right as it might not be closed under factorizations or under taking small limits. The same remark applies to $\mcall$. 
\end{rmk}

Nevertheless $\mcalk$ and $\mcall$ are closed under certain things described in  Proposition~\ref{mcalrb_prop} below. 
Before we state it, we need to introduce a functor. Consider the fibrant 
replacement functor from Section~\ref{mcalr_subsection}. 
According to Remark~\ref{utm_pt_rmk}, one has an isomorphism $\utm \cong P(\tm)$. This enables us to regard
$\mcalr_{\emptyset, P(\tm)}$ as a functor $ \ca^{\utm} \lra \ca^{\utm}$. 
We denote
\begin{eqnarray} \label{mcalrb_eqn}
\mcalrb =\mcalr_{\emptyset, P(\tm)}\colon \ca^{\utm} \lra \ca^{\utm}.  
\end{eqnarray}

\begin{prop} \label{mcalrb_prop}
	\begin{enumerate}
		\item[(i)] For every $F \in \mcalk$, $\mcalrb F$ belongs to $\mcalk$. 
		\item[(ii)]  The category $\mcall$ is  closed under taking cofibrant and fibrant replacements. 
		\item[(iii)] For every $F \in \mcalk$ (respectively $G \in \mcall$) there exists a  cylinder object $F \times I$ in $\mcalk$ (respectively $G \times I$ in $\mcall$) 
	\end{enumerate}  
\end{prop}

So it makes sense to talk about homotopy in $\mcalk$ and $\mcall$ provided that the source is cofibrant and the target is fibrant.

\begin{proof}[Proof of Proposition~\ref{mcalrb_prop}]
The statements about $\mcall$ follow since the weak equivalences in $\mcalm^{\utm}$ are pointwise weak equivalences. 
Parts (i) follows from Proposition \ref{mcalr_prop}. 	

Regarding (iii),  let $F \in \mcalk$ and let $U \in \utm$. Recalling the factorization notation (\ref{2fact}) and the closure condtions from \ref{clop}, define  $F \times I \colon \utm \lra \ca$ on objects as  $(F \times I)(U):= W_g$ where $g$ is the fold map $F(U)\coprod F(U)\rightarrow F(U)$,  and in the obvious way by naturality on morphisms. Since weak equivalences and cofibrations are both objectwise, it follows that $F \times I$ is a cylinder object for $F$. Observe $F \times I$ belongs to $\mcalk$.  
\end{proof}

Before we state and prove the main result (Proposition~\ref{fuca_prop}) of this section, we need three preparatory lemmas. We will use the notation and terminology from Section~\ref{local_category_subsection} and also Remark 
\ref{utm_pt_rmk}.

\begin{lem} \label{esur_lem}
	Consider the categories $\mcalk$ and $\mcall$ from Definition~\ref{fuca_defn}. Then the functor 
	\begin{eqnarray} \label{phi_eqn}
	\phi \colon \mcalk[\mcalw_{\mcalk}^{-1}] \lra \mcall[\mcalw_{\mcall}^{-1}], \quad F \mapsto F,
	\end{eqnarray} 
	induced by the inclusion functor $\mcalk \hra \mcall$  is essentially surjective. 
\end{lem}

\begin{proof}
	Let $n=\text{dim} M$ as above. Since the simplicial complex $\tm$ can be built up by gluing together the $\dn$'s,  it is  enough to prove the lemma when $\tm= \Delta^n$. Set $\mcalk_n = \mcalf(\mcalu(\Delta^n); \ca)$ and $\mcall_n = \mcalf_A(\mcalu(\Delta^n); \mcalm)$.  The idea of the proof is to proceed by induction on $n$ by showing that for all $n \geq 0$, for all $F \in \mcall_n[\mcalw_{\mcall_n}^{-1}]$, there exist $\gb \in \mcalk_n[\mcalw_{\mcalk_n}^{-1}]$ and a zigzag of weak equivalences $\xymatrix{F \ar[r]^-{\sim} & RF &  QRF \ar[l]_-{\sim} \ar[r]^-{\sim}_-{\beb}  & \gb}.$
	
	For $n=0$, the standard geometric simplex $\Delta^n$ has only one vertex, say $v$. So $\mcalu (\Delta^0) = \{U_v\}$.  Let $F \colon \mcalu(\Delta^0) \lra \mcalm$ be an object of $\mcall_0[\mcalw_{\mcall_0}^{-1}]$. Define $\gb \colon \mcalu (\Delta^0) \lra \ca$ as $\gb(U_v):= QRA$. By the definition of $\mcall_0$, the object $F(U_v)$ is weakly equivalent to $QRA$. So $QRF(U_v)$ is also weakly equivalent to $QRA$, that is, there exists a zigzag  of weak equivalences
	\begin{eqnarray} \label{zz_eqn}
	\xymatrix{QRF(U_v) & A_1 \ar[l]_-{\sim}^-{f_0} \ar[r]^-{\sim}_-{f_1} & A_2 & \cdots \ar[l]_-{\sim} \ar[r]^-{\sim} & A_{s-1} & A_s \ar[l]_-{\sim}^-{f_{s-1}} \ar[r]^-{\sim}_-{f_s} & QRA.}
	\end{eqnarray}
	Using standard techniques from model categories and the fact that the objects $QRF(U_v)$ and $QRA$ are both fibrant and cofibrant, one can replace (\ref{zz_eqn}) by a direct morphism $\beb[U_v] \colon QRF(U_v) \stackrel{\sim}{\lra} QRA= \gb(U_v)$. This proves the base case.

	Assume that the statement  is true for all $k \leq n-1$, and  let $F \in \mcall_n[\mcalw_{\mcall_n}^{-1}]$. We need to find $\gb \in \mcalk_n[\mcalw_{\mcalk_n}^{-1}]$ and a weak equivalence $\beb \colon QRF \stackrel{\sim}{\lra} \gb$. First of all let $\Delta^n = \vzn$, and define $\partial \mcalu(\Delta^n) \subseteq \mcalu(\Delta^n)$ as the full subposet whose objects are $U_{\sigma}$'s with $\sigma$ be a simplex of the boundary of $\Delta^n$. 
	By the induction hypothesis  there exist an  isotopy functor $G \colon \partial \mcalu(\Delta^n) \lra \ca$ 
	and a natural transformation $\beta \colon QRF|\partial \mcalu (\Delta^n) \stackrel{\sim}{\lra} G$. From the base case, there is a weak equivalence $g \colon QRF(U_{\vzn}) \stackrel{\sim}{\lra} QRA$. Since $QRA$ and $QRF(U_{\vzn})$ are both fibrant and cofibrant, it follows by Proposition~\ref{ho_prop} that $g$ admits a homotopy inverse, say $f$. Let 
	$
	H \colon QRF(U_{\vzn}) \times I \lra QRF(U_{\vzn})
	$
	denote a homotopy from $fg$ to $id$. Now define  $G' \colon \mcalu(\Delta^n) \lra \ca$ as 
	\[
	G'|\partial \mcalu(\Delta^n) := G,  G'(U_{\vzn}) := QRA,
	\text{ and }
	G'(d^i) := \beta[U_{\vzin}] \circ F(d^i) \circ f,
	\]
	where $d^i$ is the map from (\ref{di_eqn}). On the compositions we define $G'$ in the obvious way (that is, $G'(a \circ b) := G'(b) \circ G'(a)$). It is straightforward to see that $G'$ is  a contravariant functor.  
	
	One may take $\gb$ to be $G'$ and $\beb$ to be $\beta'$, where $\beta' \colon F \lra G'$ is defined as $\beta$ on $\partial \mcalu(\Delta^n)$ and $g$ on $U_{\vzn}$. The issue with that definition is the fact that the square involving $G'(d^i)$, $F(d^i)$, $g$ and $\beta[U_{\vzin}]$  is only commutative up to homotopy. To fix this, consider  the following commutative diagram.
	\[
	\xymatrix{QRF(U_{\vzn}) \ar[rr]^-{g}_-{\sim}   \ar@{>->}[dd]_-{i_0}^-{\sim} &  & G'(U_{\vzn}) \ar[dd]_-{p} \ar@{>->}[rd]^-{\sim}_{\tau} &   \\
		&   &                       &  Z_{(A,G',p)}   \ar@{->>}[ld]^-{\pb}  \\
		QRF(U_{\vzn})  \times I \ar[rr]_-{\mcalh} \ar@{.>}[rrru]^-{\overline{\mcalh}}  &  &   \underset{U_{\sigma} \in \partial \mcalu(\Delta^n)}{\text{lim}} G'(U_{\sigma})  &  }
	\]
	In that diagram  the square involving $g, \mcalh, i_0$ and $p$ is induced by the commutative square
	\[
	\xymatrix{QRF(U_{\vzn}) \ar[rrrr]^-{g}  \ar@{>->}[d]_-{i_0}  & & & &  G'(U_{\vzn}) \ar[d]_-{\sim}^-{G'(d^i)}  \\
		QRF(U_{\vzn}) \times I \ar[rrrr]_-{\beta[U_{\vzin}]F(d^i)H} & & & & G'(U_{\vzin}),}
	\] 
	while $\pb \tau$ is the factorization of $p$ such that $Z_{(QRA, G', p)}$ belongs to $\ca$, and  $\overline{\mcalh}$ is given by the lifting axiom. Now define $\gb \colon \mcalu(\Delta^n) \lra \ca$ as 
	\[
	\gb|\partial \mcalu(\Delta^n):= G', \quad  \gb(U_{\vzn}):= Z_{(QRA, G', p)} \quad \text{and} \quad \gb(d^i):= P_{U_{\sigma_i}} \circ \pb,
	\] 
	where 
	$
	\sigma_i:= \vzin$ and $P_{U_{\sigma_i}} \colon  \underset{U_{\sigma} \in \partial \mcalu(\Delta^n)}{\text{lim}} G'(U_{\sigma}) \lra G'(U_{\sigma_i})
	$
	is the canonical projection. Also define 
	\[
	\beb|\partial \mcalu(\Delta^n) := \beta \qquad \text{and} \qquad    \beb[U_{\vzn}] := \overline{\mcalh} \circ i_1,
	\]
	where $i_1 \colon QRF(U_{\vzn}) \lra QRF(U_{\vzn}) \times I$ is the canonical inclusion. 
	It is straightforward to see that $\beb$ is a weak equivalence. This proves the lemma. 
\end{proof}

\begin{lem} \label{full_lem}
	Let $\mcalk$ and $\mcall$ be as in Lemma~\ref{esur_lem}. Then the functor $\phi$ from (\ref{phi_eqn}) is full.
\end{lem}

\begin{proof}
	Let $F_1, F_2 \in \mcalk[\mcalw_{\mcalk}^{-1}]$. We need to show that the canonical map 
	\begin{eqnarray} \label{phii_eqn}
	\phi_{F_1F_2} \colon \underset{\mcalk[\mcalw_{\mcalk}^{-1}]}{\text{Hom}} (F_1, F_2) \lra \underset{\mcall[\mcalw_{\mcall}^{-1}]}{\text{Hom}} (\phi(F_1), \phi(F_2))
	\end{eqnarray} 
	induced by $\phi$ is surjective. To do this, let $f_1 \colon QF_1 \stackrel{\sim}{\lra} F_1$ be a cofibrant replacement of $F_1$. One can take $QF_1 = F_1$ and $f_1 =id$ since all objects in $\ca$ are cofibrant and by Proposition~\ref{model_prop} in the injective model category cofibrations are pointwise cofibrations. 
	
	Also let $f_2 \colon F_2 \stackrel{\sim}{\lra} \mcalrb F_2$ be a fibrant replacement of $F_2$, which lies in $\mcalk$ thanks to Proposition~\ref{mcalrb_prop}. Consider the following commutative square.
	\begin{eqnarray} \label{full_eqn}
	\xymatrix{\underset{\mcalk[\mcalw_{\mcalk}^{-1}]}{\text{Hom}} (F_1, F_2) \ar[rr]^-{\phi_{F_1F_2}} &  &  \underset{\mcall[\mcalw_{\mcall}^{-1}]}{\text{Hom}} (\phi(F_1),\phi(F_2))   \\
		\underset{\mcalk}{\text{Hom}} (F_1, \mcalrb F_2) \ar@{=}[r] \ar[u]^-{\theta} &  \underset{\mcall}{\text{Hom}} (F_1, \mcalrb F_2) \ar[r]_-{\pi} & \underset{\mcall}{\text{Hom}} (F_1, \mcalrb F_2) \slash \sim \ar[u]_-{\cong}^-{\theta'} }
	\end{eqnarray}
	In that square $\theta$ is defined as the string $\theta(f):= (f, f_2^{-1})$, $\pi$ is the canonical surjection, and the isomorphism $\theta'$ comes from Proposition~\ref{ho_prop}. The equality comes from the fact that $\mcalk$ is a full subcategory of $\mcall$ by definition. Since the composition $\theta' \circ \pi$ is surjective, and since the square commutes, it follows that $\phi_{F_1F_2}$ is a surjective map. This ends the proof. 
\end{proof}

\begin{lem} \label{ffull_lem}
	Let $\mcalk$ and $\mcall$ be as in Lemma~\ref{esur_lem}. Then the functor $\phi$ from (\ref{phi_eqn}) is faithful. 
\end{lem}

\begin{proof}
	Let $F_1, F_2 \in \mcalk[\mcalw_{\mcalk}^{-1}]$. As we mentioned in the proof of Lemma~\ref{full_lem}, $QF_1 = F_1$ and $\mcalrb F_2$ belongs to $\mcalk$. We need to show that the map $\phi_{F_1F_2}$ from (\ref{phii_eqn}) is injective. Consider the diagram (\ref{full_eqn}), and  let $\eta_0, \eta_1 \in \underset{\mcalk[\mcalw_{\mcalk}^{-1}]}{\text{Hom}} (F_1, F_2)$ such that $\phi_{F_1F_2}(\eta_0) = \phi_{F_1F_2}(\eta_1)$. This latter equality implies that $\eta'_0$ is homotopic to $\eta'_1$, where $\eta'_i := \theta'^{-1} (\phi_{F_1F_2}(\eta_i))$. By Propositions~\ref{htpy_prop}, \ref{mcalrb_prop}, there is a left homotopy $H \colon F_1 \times I \lra \mcalrb F_2$  for some  cylinder object $F_1 \times I \in \mcall$ for $F_1$. The key point of the proof is the fact that one can always choose $F_1 \times I$ in $\mcalk$ thanks to Proposition~\ref{mcalrb_prop}.  Now consider the following commutative diagram in $\mcalk$. 
	\[
	\xymatrix{F_1  \ar[rrd]^-{\eta'_0} \ar[d]_-{i_0}  &  &  \\
		F_1 \times I \ar[rr]^-{H}     &   &   \mcalrb F_2   \\
		F_1  \ar[rru]_-{\eta'_1} \ar[u]^-{i_1}}
	\] 
	Since $i_0=i_1$ in $\mcalk[\mcalw_{\mcalk}^{-1}]$  by Proposition~\ref{izio_prop}, it follows that $\eta'_0 = \eta'_1$ in $\mcalk[\mcalw_{\mcalk}^{-1}]$, which implies  $\eta_0 = \eta_1$. This proves the lemma. 
\end{proof}

\begin{prop}  \label{fuca_prop}
	Let $\mcalk$ and $\mcall$ be as in Definition~\ref{fuca_defn}. Then the inclusion $\mcalk \hra \mcall$ induces an equivalence of categories between the  localizations $\mcalk[\mcalw_{\mcalk}^{-1}]$ and $\mcall[\mcalw_{\mcall}^{-1}]$. 
\end{prop}

\begin{proof}
	This follows immediately from Lemmas~\ref{esur_lem}, \ref{full_lem} and \ref{ffull_lem} as it is well known that a  fully faithful and essentially surjective functor is an equivalence of categories. 
\end{proof}

\section{The map $\lamb \colon \futma \slash \textgoth{w} \lra  [\tmd, \ahd]$  } \label{lambda_section}

Recall the category $\fuca$ (which was called $\mcalk$ for convenience in Section~\ref{fautm_futma_subsection}) from Definition~\ref{fuca_defn}. It turns out that it is a small category since $\utm$ and $\ca$ are both small. In this section and the next one, we will view $\fuca$ as a set. 
By a well-ordering on a set, we mean a total ordering such that every non-empty subset
has a smallest element. 
Choose a well-ordering on the set of vertices of $\tm$, and let $\tmd$ denote the canonical  associated simplicial set where the simplices of $\tm$ are the non-degenerate simplices of $\tmd$.

Recall the simplicial set $\ahd$ from Section~\ref{ahd_subsection}, and let $\htah$ denote the set of simplicial maps $\tmd \lra \ahd$ and $\underset{\SsS}{\text{Hom}}(\tm, \ahd)$ the set of semi-simplicial maps.

The goal of this section is to define a map $\Lambda \colon \fuca \lra \htah$ 
and prove Proposition~\ref{lambda_prop} below, which says  that $\Lambda(F)$ is homotopic to $\Lambda(F')$ whenever $F$ is weakly equivalent to  $F'$. The notion of homotopy between two simplicial maps we use is the one  from \cite[Section I.6]{goe_jar09}. 

\begin{defn}
Let $K$ be a simplicial complex and $F\colon K\rightarrow \ca$ a weak equivalence diagram.

Define  $\overline{\Lambda}(F):= f \colon \tm \lra \ahd', \sigma \mapsto f_{\sigma}$ as follows. Let $n \geq 0,$ and let $\sigma = \langle v_0, \cdots, v_n \rangle \in \tm_n$. 
	Define  $f_{\sigma} \colon \dtn \lra \ca$ on objects as 
	$
	f_{\sigma}(\{a_0, \cdots, a_s\}) := F(U_{\langle v_{a_0}, \cdots, v_{a_s}\rangle}).
	$
	On morphisms, it is enough to define $f_{\sigma}$ on the generators $d^i$'s from Remark~\ref{dtn_rmk}.	$
	\theta^i \colon U_{\langle v_{a_0}, \cdots, \widehat{v}_{a_i}, \cdots, v_{a_s}\rangle} \hra U_{\langle v_{a_0}, \cdots, v_{a_s}\rangle}
	$
	of $\utm$, define $f_{\sigma}(d^i) :=F(\theta^i)$.

\end{defn}

Recall from Section \ref{2fun} that weak equivalence diagrams out of subcategories of $\om$ are isotopy functors. 

\begin{lem}\label{plu}
Using the equivalence of Remark \ref{utm_pt_rmk}, 
the map $\overline{\Lambda}\colon \fuca \rightarrow \shtoh $
 is a bijection. 

Also $F$ is fibrant if and only if $\overline{\Lambda}(F)$ lands in $\ahd$ and so $\overline{\Lambda}$ induces a bijection between fibrant objects in $\fuca$ and $\shtah$. 

\end{lem}
\begin{proof}
We construct the inverse of the bijection as follows.
Given $f\in \shtoh$, and $\sigma\in P(\tm)$ let $F(\sigma)=f(\sigma)([|\sigma|-1])$

The second statement is similarly straightforward. 
\end{proof}

The main goal of the rest of this section is essentially to show that this bijection takes weak equivalences to homotopy equivalences. The goal of Section \ref{theta_section} is to show the bijection takes homotopy equivalences to weak equivalences.

\begin{defn}\label{lambdaxi}
Define  $\Lambda(F):= f \colon \tmd \lra \ahd, \sigma \mapsto f_{\sigma}$ when restricting to non-degenerate simplices as $\mcalrb {\overline{\Lambda}}$. Where 
$\mcalrb$ is the fibrant replacement functor from  Equation (\ref{mcalrb_eqn}). On degenerate simplices we extend the map using Lemma \ref{detbyrest}.
Observe $f_{\sigma}$ belongs to $\ahn$ and $f \colon \tmd \lra \ahd$ is a simplicial map. 
\end{defn}

Now  on the set $\fuca$ define the equivalence relation  \lq\lq $\textgoth{w} $\rq\rq{} as $F \ \textgoth{w} \ F'$ if and only if $F$ is weakly equivalent to $F'$ in the sense of Definition~\ref{defn:we_categories}. 
Let  $[\tmd, \ahd]$ be the set of homotopy classes of simplicial maps from $\tmd$ to $\ahd$. 

\begin{defn} \label{lamb_defn}
	Consider the map $\Lambda \colon \fuca \lra \htah$ and the equivalence relation \lq\lq $\textgoth{w}$\rq\rq{} we just defined. 
	Thanks to Proposition~\ref{lambda_prop} below $\Lambda$ passes to the quotient. Define $\lamb \colon \futma \slash \textgoth{w} \lra  [\tmd, \ahd]$ to be the resulting map from the quotient. 
\end{defn}

\begin{prop} \label{lambda_prop}   
	Let $F, F' \in \fuca$. Assume $F\ \textgoth{w}\ F'$, then $\Lambda(F)$ is homotopic to $\Lambda(F')$. 
\end{prop} 

The proposition is proved at the end of the section. 

\begin{defn}  \label{we_defn}
	Let $f, f' \colon \tm \lra \ahd$ be semi-simplicial maps. 
	A morphism $\beta \colon f \lra f'$ consists of a collection 
		$
		\beta = \{\beta_{\sigma} \colon f_{\sigma} \lra f'_{\sigma}\}_{ \sigma \in \tm}
		$
		of natural transformations such that 
		\begin{eqnarray} \label{cc_disk_eqn}
		\beta_{d_i\sigma}=d_i\beta_{\sigma}  \quad \text{for all $\sigma$, $i$}. 
		\end{eqnarray}
		The morphism $\beta$ is a weak equivalence if each $\beta_{\sigma}$ is a weak equivalence of diagrams. (Note this was defined at the end of the introduction.)
\end{defn} 

In the proof of the next lemma and in Lemma \ref{htpy2_lem} we let $\Delta[1]$ denote the simplicial set associated to the semi-simplicial set $\dt^1$. 

\begin{lem} \label{htpy1_lem}
	Let $f, f' \colon \tm \lra \ahd$ be semi-simplicial maps. Assume that there is a morphism $\beta \colon f \lra f'$ which is a weak equivalence (see Definition~\ref{we_defn}). Then $f$ is homotopic to $f'$. 
\end{lem}

\begin{proof}
	
	Our goal is to define a 
	homotopy $H \colon \tmd \times \Delta[1] \lra \ahd$ between $f$ and $f'$. 
	The proof breaks into two parts. 1) Construct a $\SsS$ map $H' \colon \tm \times \dt^1 \lra \ahd'$ and then 2) using fibrant replacement turn $H'$ into $H\colon  \tm \times \dt^1 \lra \ahd$, and then use Lemma \ref{detbyrest}.

	Starting with 1) the construction is similar to a mapping cone.
	The vertices  of $\tm\times \dt^1$ are pairs $(v,0)$ and $(v,1)$ where $v$ is a vertex of $\tm$. 
	A simplex $\tau\in  \tm\times \dt^1$ is an increasing sequence of pairs where the vertices of $\tm$ that appear make up a simplex
	$\sigma_{\tau}\in \tm$. Note that with this structure $(\tm\times \dt^1)_{\bullet}=\tm_{\bullet}\times \Delta[1]$. 
	If all of the vertices of $\tau$ are of the form $(v,0)$, then set $H'(\tau)=f'(\sigma_{\tau})$. 
	Otherwise set $H'(\tau)=f(\sigma_{\tau})$. The $d_i\tau$ is obtained by taking out one of the vertices in the sequence making up $\tau$. There are a few cases to describe $H'(d_i)$. If $\sigma_{\tau}=\sigma_{d_i\tau}$, then $H'(d_i)$ is 
	$id$ or $\beta(\sigma_{tau})$ if the vertices of $d_i\tau$ are all of the form $(v,0)$. If $\sigma_{\tau}\not=\sigma_{d_i\tau}$ then $H'(d_i)$ is either $d_i$ or $\beta(d_i)$ again if all the vertices are $d_i\tau$ are of the form $(v,0)$ and not all the vertices of $\tau$ are of that form. 
	So we have constructed a $\SsS$   map $H'\colon \tm \times \dt^1  \lra \ahd'$ with the 
	correct restriction to $\tm\times \{0,1\}$.

Onto step 2), we let $\cal H'$ be the functor corresponding to $H'$ using Lemma \ref{plu} and 
${\cal H}=\mcalr_{\tm\times \{0,1\}, \tm\times \dt^1 } ({\cal H'})$. By Proposition \ref{mcalr_prop} 
$\cal H$ is a fibrant functor 
$\tm\times \dt^1 \rightarrow \ca$ such that the restrictions to the end of the cylinder are (the functors corresponding to) $f'$ and $f$. 
Again using \ref{plu} we get $H\colon  \tm \times \dt^1 \lra \ahd$ with the correct restrictions and so Lemma \ref{detbyrest} gives the simplicial homotopy between $f$ and $f'$ and 
completes the proof of the lemma.

\end{proof}

Now we can prove the main result of this section.

\begin{proof}[Proof of Proposition~\ref{lambda_prop}] 
	Let $\eta \colon F \stackrel{\sim}{\lra} F'$ be a weak equivalence of $\futma$. Set $\Lambda(F) = f$ and  
	$\Lambda(F') = f'$. For $n \geq 0$ and $\sigma = \langle v_0, \cdots, v_n \rangle \in \tm$,  define $\beta_{\sigma} \colon f_{\sigma} \lra f'_{\sigma}$ as
	$
	\beta_{\sigma}[\{a_0, \cdots, a_s\}] := (\mcalrb\eta) [U_{\langle v_{a_0}, \cdots, v_{a_s}\rangle}].
	$
	By definition $\beta_{\sigma}$ is a weak equivalence for all $\sigma$. It is straightforward to check that the collection $\{\beta_{\sigma}\}_{\sigma}$ satisfies (\ref{cc_disk_eqn}). 
	
	Applying Lemma~\ref{htpy1_lem}, we have that $f$ is homotopic to $f'$. Now assume that there is  a zigzag $\xymatrix{F & \bullet \cdots \bullet \ar[l]_-{\sim} \ar[r]^-{\sim} & F'}$. Applying the first part to each map of that zigzag, and taking the inverse homotopy associated to each backward arrow, we have a homotopy between $\Lambda(F)$ and $\Lambda(F')$.  This proves the proposition. 
\end{proof}

\section{The map $\ttb \colon [\tmd, \ahd] \lra \futma \slash \textgoth{w}$}  \label{theta_section}

In Section~\ref{lambda_section}, or more precisely in Definition~\ref{lamb_defn}, we defined a map $\lamb \colon \fuca \slash \textgoth{w} \lra [\tmd, \ahd]$. The goal of this section is to construct its inverse, and thus get (\ref{res_eqn5}). We continue to use the well-ordering on the set of vertices of  $\tm$ we chose in  the previous section. 

\begin{defn}  \label{tt_defn}
	Define a map $\Theta \colon \htah \lra \fuca$ as $\Theta(f):= F$, where $F \colon \utm \lra \ca$ is defined on objects as $F(U_{\sigma}):= \fs([n])$ for $\sigma = \langle v_{a_0}, \cdots, v_{a_n}\rangle$. On morphisms, it is enough to define $F$ only on  $d^i$'s from Remark~\ref{utm_rmk}. Define 
	$
	F(d^i) := f_{\sigma}(d^i) \colon f_{\sigma}([n]) \lra f_{\sigma}([n]_i).
	$
\end{defn}

\begin{defn} \label{ttb_defn}
	Thanks to  Proposition~\ref{theta_prop} below, the map  $\Theta$ 
	passes to the quotient. Define $\ttb \colon  [\tmd, \ahd] \lra \futma \slash \textgoth{w}$ to be the resulting quotient map. 
\end{defn}   

\begin{prop}  \label{theta_prop}
	Let $f, f' \colon \tmd \lra \ahd$ be two simplicial maps such that $f$ is homotopic to $f'$. Then there exists a zigzag of weak equivalences  $\xymatrix{\Theta(f) & \bullet \cdots \bullet \ar[l]_-{\sim} \ar[r]^-{\sim} & \Theta(f')}$ in $\fuca$.  
\end{prop}

The proof of Proposition~\ref{theta_prop} occupies the section and will be given at the end. Our strategy goes through two big steps. In the first one  
we prove the result when the poset $\utm$ is finite.  In the case where $\utm$ is infinite,  the idea of the proof is to write $F=\Theta(f)$ as 
the homotopy limit of a certain diagram $E_1F_1 \lla E_2F_2 \lla \cdots$, where $F_i$ is the restriction of $F$ to a finite subposet $\utmi \subseteq \utm$ 
and $E_iF_i$ is the right Kan extension of $F_i$ along the inclusion $\utmi \hra \utm$.  Using the fact that $F_i$ is weakly equivalent to $F'_i$ by the first step, 
one can deduce the proposition.  

This section is organized as follows. In Section~\ref{finite_case_subsection} we prove Lemma~\ref{htpy2_lem} below, which says that if $f, f' \colon \tmd \lra \ahd$ 
are homotopic, then $\Theta(f)$ is weakly equivalent to $\Theta(f')$ provided that $\tm$ has finite number of simplices. Since the proof of that lemma is technical,  
we begin with a special case: $\tm = \Delta^1$. Section~\ref{infinite_case_subsection} deals with the case where $\tm$ has infinite number of simplices.

\subsection{Case where $\tm$ has finite number of simplices}    \label{finite_case_subsection}

\sloppy

\begin{lem} \label{htpy1-1_lem} 
	Let $\tm = \Delta^1$, and let $f, f' \colon \tmd \lra \ahd$ be simplicial maps that are homotopic.  Then there exists a zigzag of weak equivalences  $\xymatrix{\Theta(f) & \bullet \cdots \bullet \ar[l]_-{\sim} \ar[r]^-{\sim} & \Theta(f')}$ in $\fuca$.
\end{lem}

\begin{proof}
	Let $H$ be a homotopy between $f$ and $f'$.  We only consider the homotopy restricted to non-degenerate simplices. 
	For the sake of simplicity,  we will often use the notation $v_{0 \cdots n} := \langle v_0, \cdots, v_n \rangle$ in this proof and the next one. Consider the poset  $\udo = \{ \xymatrix{ U_{v_0} \ar[r]^-{d^1}    & U_{v_{01}}   &  U_{v_1} \ar[l]_-{d^0}}\}$. Also consider  Figure~\ref{sdt_splx}, which is a subdivision of $\Delta^1 \times [0, 1]$ into two $2$-simplices, namely 
	\[
	\langle (v_0, 0),  (v_0, 1), (v_1, 1)\rangle \quad \textup{and} \quad \langle (v_0, 0), (v_1, 0), (v_1, 1)\rangle
	\]
	\begin{figure}[!ht]
		\centering
		\includegraphics[scale = 0.7]{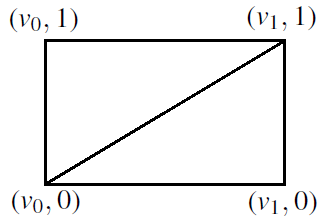}
		\caption{A subdivision of $\Delta^1 \times [0, 1]$ into two $2$-simplices} \label{sdt_splx}
	\end{figure}
    We now explain the algorithm that produces functors out of $H$ and Figure~\ref{sdt_splx}. 
	Define the \textit{barycenter} of $\sigma_0:= \langle (v_0, 0), (v_1, 0), (v_1, 1)\rangle$ as the pair $(v_{011}, 001)$, and that of $\sigma_1:=\langle (v_0, 0),  (v_0, 1), (v_1, 1)\rangle$ as the pair $(v_{001}, 011)$. In general, the barycenter of $\langle (v_{i_0}, j_0), (v_{i_1}, j_1), (v_{i_2}, j_2)\rangle$ is defined to be $(v_{i_0i_1i_2}, j_0j_1j_2)$.  Applying the homotopy $H$ to those barycenters, we get the commutative diagram (\ref{big_eqn}) below in which we make the following simplifications at the level of notation. We write $H(v_{011}, 001)$ for $H(v_{011}, 001)(\{0, 1,2\})$, $H(v_{01}, 01)$ for $H(v_{01}, 01)(\{0,1\})$, and so on. Also we write $d_i$ for $H(-, -)(d^i)$. 
	
	\begin{eqnarray} \label{big_eqn}
	\xymatrix{H(v_0, 1)   &  & H(v_{01}, 11)  \ar[ll]_-{d_1}^-{\sim}  \ar[rr]^-{d_0}_-{\sim}   &  & H(v_1, 1)  \\
		&     &   H(v_{001}, 011) \ar[u]^-{\sim}_-{d_0}  \ar[lld]_-{d_2}^-{\sim}  \ar[d]^-{d_1}_-{\sim}   &      &       \\
		H(v_{00}, 01)  \ar[uu]^-{d_0}_-{\sim} \ar[dd]^-{\sim}_-{d_1}   &     &   H(v_{01}, 01)  \ar[rruu]^-{d_0}_-{\sim}  \ar[lldd]^-{\sim}_-{d_1}   &    &  H(v_{11}, 01)  \ar[uu]^-{\sim}_-{d_0} \ar[dd]^-{d_1}_-{\sim}    \\
		&     &   H(v_{011}, 001)  \ar[u]^-{\sim}_-{d_1}  \ar[d]^-{d_2}_-{\sim}  \ar[rru]^-{\sim}_-{d_0}   &    &    \\
		H(v_0, 0)      &        &   H(v_{01}, 00)  \ar[ll]_-{\sim}^-{d_1} \ar[rr]^-{\sim}_-{d_0}   &      &  H(v_1, 0)   }
	\end{eqnarray}
	
	Now define the bottom  of (\ref{big_eqn}) as a functor $F_0^L \colon \udo \lra \ca$, where the letter \lq\lq$L$\rq\rq{} stands for \lq\lq lower\rq\rq{}. Specifically, we have 
	\[
	F_0^L (U_{v_0}) := H(v_0, 0), \  F_0^L (U_{v_{01}}) := H(v_{01}, 00), \  F_0^L (U_{v_1}) := H(v_1, 0), \ \text{and} \ F_0^L (d^i):= d_i. 
	\] 
	Observe
	that functor is the same as $F$. The functor associated to the barycenter of the simplex $\sigma_0$ is defined as (\lq\lq B\rq\rq{} stands for \lq\lq barycenter\rq\rq{})
	\[
	F_0^B := \left\{ \xymatrix{H(v_{0}, 0)    &    &  H(v_{011}, 001)  \ar[ll]_-{d_1d_2}^-{\sim} \ar[rr]^-{d_0}_-{\sim}  &  & H(v_{11}, 01) } \right\},  
	\]
	while the functor corresponding to its upper face (which is the same as the functor associated to the lower face of $\sigma_1$) is defined as
	\[
	F_0^U = F_1^L := \left\{ \xymatrix{H(v_{0}, 0)    &    &  H(v_{01}, 01)  \ar[ll]_-{d_1}^-{\sim} \ar[rr]^-{d_0}_-{\sim}  &  & H(v_1, 1) } \right\}.
	\]
	Here \lq\lq U\rq\rq{} stands for \lq\lq upper\rq\rq{} of course. Lastly, the functor corresponding to the barycenter of the $2$-simplex $\sigma_1$ is defined as 
	\[
	F_1^B := \left\{ \xymatrix{H(v_{00}, 01)    &    &  H(v_{001}, 011)  \ar[ll]_-{d_2}^-{\sim} \ar[rr]^-{d_0d_0}_-{\sim}  &  & H(v_1, 1) } \right\},
	\]
	and that associated to its upper face, denoted $F_1^U$, is defined as the top of (\ref{big_eqn}).  Clearly one has the  following zigzag of  weak equivalences, which are all natural since (\ref{big_eqn}) is commutative. 
	\[
	\xymatrix{F= F_0^L  & &   F_0^B \ar[ll]_-{(id, d_2, d_1)}^-{\sim} \ar[rr]^-{(id, d_1, d_0)}_-{\sim} & & F_0^U=F_1^L & &  F_1^B   \ar[ll]_-{(d_1, d_1, id)}^-{\sim} \ar[rr]^-{(d_0, d_0, id)}_-{\sim} & & F_1^U =F'.}
	\]
	This ends the proof. 
\end{proof}

\begin{lem} \label{htpy2_lem} 
	Let $f, f' \colon \tmd \lra \ahd$ be simplicial maps that are homotopic. Assume  that  $\tm$ has finite number of simplices. Then there exists a zigzag of weak equivalences  $\xymatrix{\Theta(f) & \bullet \cdots \bullet \ar[l]_-{\sim} \ar[r]^-{\sim} & \Theta(f')}$ in $\fuca$.
\end{lem}

\begin{proof}
	Since the simplicial complex $\tm$ has finite number of simplices, one can  assume without loss of generality that  it is a subcomplex of $\Delta^n$ for some $n$. 
	Set $F = \Theta(f)$ and $F' = \Theta(f')$, and let $H \colon \tmd \times \Delta[1] \lra \ahd$ be a homotopy from $f$ to $f'$.  
	We need to construct a zigzag of weak equivalences between the  functors $F, F' \colon \utm \lra \ca$. Throughout we are only using the non-degenerate simplices in $\tmd\times \Delta[1]$. Following the special case above, 
	the idea is to first subdivide $\dn \times [0, 1]$ into $(n+1)$ $(n+1)$-simplices in a suitable way.  
	Each $(n+1)$-simplex will produce three functors (one for the upper face, one for the lower face, and one for the barycenter), and two natural transformations like $\xymatrix{\bullet & \bullet \ar[l]_-{\sim} \ar[r]^-{\sim}  & \bullet}$. 
	
	Let us consider the prism $\dn \times I$, I = [0, 1], and set 
	\[
	\dn \times \{t\} = \langle (v_0, t), \cdots, (v_n, t) \rangle, \quad t \in \{0, 1\}.
	\]
	We will sometimes write $v_j$ for $(v_j, 0)$. For $0 \leq i \leq n$, define an $(n+1)$-simplex 
	\[
	\sigma_i := \big\langle (v_0, 0), \cdots, (v_{n-i}, 0), (v_{n-i}, 1), \cdots, (v_n, 1)\big\rangle.
	\]
	Observe that 
	$\dn \times [0, 1]$ is the union of (geometric) simplices $\sigma_i, 0 \leq i \leq n$, each intersecting the next in an $n$-simplex face. This induces a triangulation $\dt\times I$ on $M\times I$. 
	As we said earlier, each $\sigma_i$ produces three functors $F_i^L,  F_i^B, F_i^U \colon \utm \lra \ca$ that we now define. 
	Consider the diagram 
	\[ 
	\xymatrix{\utm \ar[r]^-{\Psi_{\tm}}_-{\cong} & P(\tm) \ar@<3ex>[r]^-{\phi_i^U} \ar@<0ex>[r]^-{\phi_i^B} \ar@<-3ex>[r]_-{\phi_i^L}  & P(\sigma_i) \cap P(\tm \times I) \  \ar@{^{(}->}[r]^-{\iota} & P(\tm \times I) \ar[r]^-{\Phi}_-{\cong} & \mcalu(\tm \times I)  \ar[d]^-{\overline{H}} \\
		&  &   &   & \ca}
	\]
	defined as follows.

	$\bullet$ $P(-)$ is taking the poset associated to a triangulation. 
	
	$\bullet$ The functor $\Psi_{\tm}$ is defined as $\Psi_{\tm}(U_{\lambda}):= \lambda$. As discussed above Remark 
	\ref{utm_pt_rmk} 
	this is an isomorphism of categories. The functor $\Phi$ is defined as $\Phi:= \Psi_{\tm \times I}^{-1}$. 
	
	$\bullet$ $\iota$ is just the inclusion functor. 
	
	$\bullet$ Defining $\overline{H}$. Recall the map $\Theta$ from Definition~\ref{tt_defn}, and let us denote it by $\Theta_{\tm}$. Then one has the  map 
	\[
	\Theta_{\tm \times I} \colon \underset{\textup{sSet}}{\textup{Hom}}((\tm \times I)_{\bullet}, \ahd) \lra \mcalf(\mcalu(\tm \times I); \ca). 
	\]
	The functor $\overline{H}$ is defined as $\overline{H}:= \Theta_{\tm \times I} (H)$.
	
	$\bullet$ We now define the functor $\phi^B_i \colon P(\tm) \lra P(\sigma_i) \times P(\tm \times I)$. Let $\lambda=\langle v_{a_0}, \cdots, v_{a_s}\rangle$ be an object of $P(\tm)$. We need to deal with four cases.
	\begin{enumerate}
		\item[-] If  $n-i = a_j$ for some $j$, define 
		\[
		\phi^B_i(\lambda):= \big\langle (v_{a_0}, 0), \cdots, (v_{a_j}, 0), (v_{a_j}, 1), (v_{a_{j+1}}, 1), \cdots, (v_{a_s}, 1)\big\rangle.
		\]
		\item[-] If $a_j < n-i < a_{j+1}$ for some $j$, define
		\[
		\phi^B_i(\lambda):= \big\langle (v_{a_0}, 0), \cdots, (v_{a_j}, 0), (v_{a_{j+1}}, 1), \cdots, (v_{a_s}, 1)\big\rangle.
		\] 
		\item[-] If $n-i < a_j$ for all $j$, define 
		\[
		\phi^B_i(\lambda):= \big\langle (v_{a_0}, 1), \cdots, (v_{a_s}, 1)\big\rangle.
		\]
		\item[-] If $n-i > a_j$ for all $j$, define 
		\[
		\phi^B_i(\lambda):= \big\langle (v_{a_0}, 0), \cdots, (v_{a_s}, 0)\big\rangle. 
		\]
	\end{enumerate}   
	$\bullet$ Define $\phi^U_i$ as  
	\[
	\phi^U_i(\lambda):= \big\langle (v_{a_0}, 0), \cdots, \widehat{(v_{a_j}, 0)}, (v_{a_j}, 1), (v_{a_{j+1}}, 1), \cdots, (v_{a_s}, 1)\big\rangle 
	\] 
	if $n-i=a_j$ for some $j$, and $\phi^U_i (\lambda):= \phi^B_i(\lambda)$ otherwise. 
	
	$\bullet$ Define $\phi^L_i$ as  
	\[
	\phi^L_i(\lambda):= \big\langle (v_{a_0}, 0), \cdots, (v_{a_j}, 0), \widehat{(v_{a_j}, 1)}, (v_{a_{j+1}}, 1), \cdots, (v_{a_s}, 1)\big\rangle 
	\] 
	if $n-i=a_j$ for some $j$, and $\phi^L_i (\lambda):= \phi^B_i(\lambda)$ otherwise. 
	
	For $X \in \{B, L, U\}$, define $F^X_i$ as the composite
	\[
	F^X_i:= \overline{H} \circ  \iota \circ \phi^X_i \circ \psi.
	\]
	From the definitions, there are two natural transformations $\alpha_i \colon \phi^L_i \lra \phi^B_i$ and $\beta_i \colon \phi^U_i \lra \phi^B_i$ which are nothing but the inclusions (that is, for every $\lambda \in P(\tm)$, the components of $\alpha_i$ and $\beta_i$ at $\lambda$ are the obvious inclusions). These give rise to a zigzag $F^L_i \lla F^B_i \lra F^U_i$ (remember $\overline{H}$ is contravariant). Each map from that zigzag is a weak equivalence since the functor $\overline{H}$ sends every morphism to a weak equivalence by the definitions. Clearly, one has $F^L_0 = F, F^U_n = F'$, and $F^U_i = F^L_{i+1}$ for all $i$. This ends the proof. 
\end{proof}

\subsection{Case where $\tm$ has infinite number of simplices}   \label{infinite_case_subsection}

In this section we assume that $\tm$ has infinitely many simplices. Since $M$ is second-countable by assumption (see the introduction), 
it follows that $\tm$ has countably many simplices. This enables us to choose a family 
\[
\mcalt^{M_1} \subseteq  \cdots \subseteq \mcalt^{M_i}  \subseteq \mcalt^{M_{i+1}} \subseteq  \cdots \subseteq \tm 
\] 
of subcomplexes of $\tm$ satisfying the following two conditions: 
\begin{enumerate}
	\item[(a)] Each $\mcalt^{M_i}$ has finitely many simplices; 
	\item[(b)] $P(\mcalt^{M_{i+1}}) = P(\mcalt^{M_i}) \cup \{z\}$, where $z \notin P(\mcalt^{M_i})$  and for every simplex $x$ of  $\mcalt^{M_i}$, $z$ is not a face of $x$.  
\end{enumerate}
Let $\{\utmi\}_{i \geq 1}$ be the corresponding family of subposets of $\utm$. 
(Observe one has $\bigcup_i \utmi = \utm$.) Let $R_i \colon \utmi \hra \utm$ and $R_{ij} \colon \utmj \hra \utmi$ denote the inclusion functors. Consider the pair 
\[
R_i^* \colon \xymatrix{\mcalm^{\utm} \ar@<1ex>[rr] &  & \mcalm^{\utmi} \ar@<1ex>[ll]}: E_i,
\]
where $R_i^*$ is the restriction functor (that is, $R_i^* (F):= F|\utmi$), and $E_i$ is nothing but the right Kan extension functor along $R_i$, that is, 
\begin{eqnarray}  \label{ei_eqn}
E_i(F)(x):= \underset{y \ra x, y \in \utmi}{\lim} F(y) = \underset{\utmi \downarrow x}{\lim} \mathbb{D}_i.
\end{eqnarray}
Here $\mathbb{D}_i \colon \utmi \downarrow x \lra \mcalm$ is the functor that sends $(y \lra x)$ to $F(y)$. Clearly $E_i$ is right adjoint to $R_i^*$.  Similarly, we have a diagram $\mathbb{D}_{ij}$, and a pair of adjoint functors
\[
R_{ij}^* \colon \xymatrix{\mcalm^{\utmi} \ar@<1ex>[rr] &  & \mcalm^{\utmj} \ar@<1ex>[ll]}: E_{ij},
\]
where $E_{ij}(F)(x)$ is of course the limit of $\mathbb{D}_{ij}$. 

Before we prove Proposition~\ref{theta_prop}, we need five lemmas. 

\begin{lem}  \label{facts_lem}
	Let $j \leq i$. Then 
	\begin{enumerate}
		\item[(i)] For every $F \in \mcalm^{\utmj}$, $E_i (E_{ij} (F))$ is naturally isomorphic to $E_j(F)$. That is, $E_i (E_{ij} (F)) \cong E_j(F)$. 
		\item[(ii)] Let $F \in \mcalm^{\utm}$ be a fibrant diagram. 
		Then the diagram $\mathbb{D}_i$  above as well as $\mathbb{D}_{ij}$ is fibrant for any $x \in \utm$.  
	\end{enumerate}
\end{lem}

\begin{proof}
	Part (i) follows immediately from the definitions. To see part (ii), let $x \in \utm$. Then, by Definition~\ref{utm_defn}, there exists a unique simplex $\sigma_x \in \tm$ such that $x = U_{\sigma_x}$. Associated with $\sigma_x$ is the poset $\mcalu (\sigma_x) = \{U_{\lambda}| \ \lambda \subseteq \sigma_x\}$. Clearly we have $\utmi \downarrow x = \utmi \cap \mcalu(\sigma_x)$. So  $\mathbb{D}_i$ is nothing but the restriction of $F$ to  $\utmi \cap \mcalu(\sigma_x)$, which is definitely fibrant since $F$ is fibrant by assumption. Likewise, one can show that $\mathbb{D}_{ij}$ is fibrant. 
\end{proof}

Let $\Theta$ be the map that appears in Proposition~\ref{theta_prop}, and let $f, f' \colon \tmd \lra \ahd$ be simplicial maps that are homotopic. Consider the categories $\mcalm$ and $\ca$ as in Notation~\ref{CAnotation}, 
and set $F = \Theta(f)$ and $F'= \Theta (f')$. By the definitions, the functors $F, F' \in \ca^{\utm} \subseteq \mcalm^{\utm}$ are both fibrant and cofibrant. For convenience, we will regard these as functors into $\mcalm$. 
Define $F_i := F|\utmi$ and $F'_i := F'|\utmi$. 

\begin{lem} \label{fib_lem}
	For every $i$ there exist $\fib \in \mcalm^{\utmi}$ and two weak equivalences $\xymatrix{F_i  & \fib \ar[l]_-{\alpha_i}^-{\sim}  \ar[r]^-{\alpha'_i}_-{\sim}   &   F'_i}$  satisfying the following two conditions: (a) $\fib$ is cofibrant. (b) The map $(\alpha_i, \alpha'_i) \colon \fib \lra F_i \times F'_i$ is a fibration.  
\end{lem}

\begin{proof}
	Consider the map $\Theta_i \colon \underset{\text{sSet}}{\text{Hom}} (\mcalt^{M_i}_{\bullet}, \ahd) \lra \mcalf(\utmi; \ca)$ defined in the same way as $\Theta$, and let $f_i = f|\mcalt^{M_i}_{\bullet}$ and $f'_i = f'|\mcalt^{M_i}_{\bullet}$. Then it is clear that $f_i$ is homotopic to $f'_i$, $\Theta_i(f_i) = F_i$ and $\Theta_i(f'_i) = F'_i$. Applying Lemma~\ref{htpy2_lem}, we get a zigzag 
	\begin{eqnarray} \label{fi_zigzag}
	\xymatrix{F_i & \bullet \ar[l]_-{\sim} \ar[r]^-{\sim} & \cdots   & \bullet \ar[l]_-{\sim} \ar[r]^-{\sim} & F_i'}
	\end{eqnarray}
	Using now the fact that  $F_i$ and $F_i'$ are both fibrant and cofibrant, and some standard techniques from model categories, one can construct out of (\ref{fi_zigzag}) a short zigzag (of the form indicated in the lemma) which has the required properties.
\end{proof}

\begin{lem}
	Let $j \leq i$, and let $\alpha_i$ and $\alpha'_i$ be as in Lemma~\ref{fib_lem}. Then there exists a map $\fjb  \lla R_{ij}^* \fib$ making the following diagram commute.
	\begin{eqnarray} \label{fi_diag}
	\xymatrix{R^*_{ij}F_i \ar[d]_-{id} & & R^*_{ij} \fib \ar[ll]_-{\sim}^-{R^*_{ij}(\alpha_i)} \ar[rr]^-{\sim}_-{R^*_{ij}(\alpha'_i)} \ar[d]^-{\sim} &  & R^*_{ij} F'_i  \ar[d]^-{id}  \\
		F_j    &  & \fjb \ar[ll]_-{\sim}^-{\alpha_j} \ar[rr]^-{\sim}_-{\alpha'_j} & &  F'_j.} 
	\end{eqnarray}
\end{lem}

\begin{proof}
	Since $\alpha_j$ is a fibration by Lemma~\ref{fib_lem}, and since $F_j$ is cofibrant, by the lifting axiom, there exists a map $\alpha_j^{-1} \colon F_j \lra \fjb$ making the obvious square commute. From the construction of the zigzag (\ref{fi_zigzag}) (look closer at the proof of Lemma~\ref{htpy2_lem}), the following square commutes up to homotopy.
	\begin{eqnarray}   \label{fi_sq1}
	\xymatrix{R^*_{ij} F_i \ar[rrr]^-{R^*_{ij} (\alpha'_i)R^*_{ij}(\alpha_i^{-1})} \ar[d]_-{id}   &  &  &  R^*_{ij} F'_i \ar[d]^-{id} \\
		F_j \ar[rrr]_-{\alpha'_j \alpha_j^{-1}} &  &  & F'_j.}
	\end{eqnarray}
	Since $\alpha'_j$ is a fibration, and since $R^*_{ij} \fib$ is cofibrant (this is because $\fib$ is cofibrant by Lemma~\ref{fib_lem}, and cofibrations are objectwise), the lifting axiom guarantees the existence of  a map $R_{ij}^* \fib \stackrel{\phi}{\lra} \fjb $ that makes the righthand square of the following diagram commute. 
	\begin{eqnarray} \label{fi_sq2}
	\xymatrix{R^*_{ij}F_i \ar[d]_-{id} & & R^*_{ij} \fib \ar[ll]_-{\sim}^-{R^*_{ij}(\alpha_i)} \ar[rr]^-{\sim}_-{R^*_{ij}(\alpha'_i)} \ar[d]^-{\sim}_-{\phi} &  & R^*_{ij} F'_i  \ar[d]^-{id}  \\
		F_j    & & \fjb \ar[ll]_-{\sim}^-{\alpha_j} \ar[rr]^-{\sim}_-{\alpha'_j} & &  F'_j.} 
	\end{eqnarray} 
	Combining this with the fact that the square (\ref{fi_sq1}) commutes up to homotopy, we deduce that the lefthand square commutes up to homotopy as well. By Lemma~\ref{reg_lem} below,  one can then replace $\phi$ by a map $R_{ij}^* \fib  \stackrel{\phib}{\lra} \fjb$ that makes the whole diagram commute. 
\end{proof}

\begin{lem}  \label{reg_lem}
	Consider the following diagram in a model category $\mcalm$. 
	\[
	\xymatrix{A_0 \ar[d]_-{g_0}  &  B \ar[l]_-{f_0}  \ar[r]^-{f_1}  \ar[d]^-{g}  &  A_1 \ar[d]^-{g_1}  \\
		D_0  & C \ar[l]^-{f'_0}  \ar[r]_-{f'_1}  & D_1.}
	\]
	Assume that each square commutes up to homotopy. Also assume that $B$ is cofibrant. If the map $(f'_0, f'_1) \colon C \lra D_0 \times D_1$ is a fibration, then there exists $\overline{g}$ homotopic to $g$ that makes the whole diagram commute. 
\end{lem}

\begin{proof}
	Since each square commutes up to homotopy, there exists a homotopy $H \colon B \times I \lra D_0 \times D_1$ from $(g_0f_0, g_1f_1)$ to $(f'_0, f'_1) \circ g$, which fits into the following commutative diagram. (Note that here $B\times I$ denotes the cylinder object in the model category. 
	\[
	\xymatrix{    &  B \ar[r]^-{g}  \ar@{>->}[d]_-{i_1}^-{\sim}   &  C \ar@{->>}[d]^-{(f'_0, f'_1)}  \\
		B \  \ar@{>->}[r]_-{i_0}  & B \times I  \ar[r]_-{H} \ar[ru]^-{\psi}   &   D_0 \times D_1.}
	\]
	The canonical inclusion $i_1$ is an acyclic cofibration since $B$ is cofibrant \cite[Lemma 4.4]{dwyer_spa95}. The map $(f'_0, f'_1)$ is a fibration by assumption, and the map $\psi$ is provided by the lifting axiom. Now define $\overline{g} = \psi \circ i_0$. It is straightforward to check that $\overline{g}$ does the  work. 
\end{proof}

We still need an important lemma. From the definition of the right Kan extension (\ref{ei_eqn}), there is a canonical map $\eta_i \colon E_i F_i \lra E_{i-1}F_{i-1}$ induced by the universal property of  limit.  These maps fit into the covariant diagram $\ebb \colon \nbb \lra \mcalm^{\utm}$ defined by $\ebb(i) = E_iF_i$ and $\ebb(i \ra (i-1)) = \eta_i$. Here $\nbb$ is viewed as the poset $\{1 \la 2 \la 3 \la \cdots\}$. 

\begin{lem}  \label{kappa_lem}
	The canonical map $ \kappa \colon F \lra \textup{holim}_{\nbb} \ebb $ 
	is a weak equivalence. 
\end{lem}

\begin{proof}
	We begin by claiming that for every $x \in \utm$, there exists $s \in \nbb$ such that $F(x) \cong E_rF_r (x)$ for all $r \geq s$. 
	To see this, let $x \in \utm$. Since $\bigcup_i \utmi = \utm$, there exists $s \in \nbb$ such that $x \in \utms$. Let $r \geq s$.  Since the sequence $\{\utmi\}_i$ is increasing, the indexing category  $\utmr \downarrow x$ has a terminal object, namely $x \stackrel{id}{\lra} x$. This implies (by the definition (\ref{ei_eqn}) of $E_r$) that the canonical map $F(x) \lra  E_rF_r (x)$ is an isomorphism. Thanks to that isomorphism, we have $F \cong \underset{i}{\lim} \, E_i F_i$. To end the proof, it suffices to show that the diagram $\mathbb{E} \colon\nbb \lra \mcalm^{\utm}$ is fibrant. 
	By Proposition~\ref{model_prop}, we have  to show that the matching map of $\ebb$ at $i$ is a fibration for any $i$. This is the case for $i =1$ since $E_1F_1$ is fibrant (because of the fact that $F$ is fibrant). Let $i \geq 2$. Then the matching map at $i$ is nothing but the canonical map $\eta_i \colon E_i F_i \lra E_{i-1} F_{i-1}$. Looking at the definition of a fibration (see Proposition~\ref{model_prop}) in $\mcalm^{\utm}$, we need to show that the map $p_x$ from the following commutative diagram is a fibration for every  $x \in \utm$ to conclude that $\eta_i$ is a fibration. 
	\[
	\xymatrix{E_iF_i (x) \ar[rd]^-{p_x} \ar[rrd]^-{(\eta_i)_x} \ar[rdd]_-{\theta_{xi}}  &   &   \\
		&  \text{PB} \ar[d]^-{\lambda_x} \ar[r]_-{\beta_x}  &  E_{i-1}F_{i-1} (x) \ar[d]^-{\theta_{x(i-1)}}\\
		& \underset{y \ra x, y \neq x}{\lim} E_iF_i (y) \ar[r]^-{M_x(\eta_i)}  &  \underset{y \ra x, y \neq x}{\lim} E_{i-1}F_{i-1} (y)  }
	\]
	So let $x \in \utm$. Let $z \in \utm$ such that $\utmi = \utmii \cup \{z\}$. We need to deal with two cases depending on the fact that $x = z$ or $x \neq z$. 
	\begin{enumerate}
		\item[$\bullet$] If $x = z$, then $\theta_{x(i-1)}$ is an isomorphism. Since the pullback of an isomorphism is an isomorphism, it follows that $\lambda_x$ is an isomorphism. Since $\theta_{xi}$ is exactly the matching map of $E_iF_i$ at $z =x$, the map $\theta_{xi}$ is a fibration. Thus $p_x$ is a fibration. 
		\item[$\bullet$] Assume that $x \neq z$.  We have two cases. If there is no map from $z$ to $x$, then by the definitions, we have $E_{i-1}F_{i-1} (y) = E_iF_i (y)$ for all $y \ra x$. This implies that $M_x(\eta_i), \beta_x, (\eta_i)_x,$ and $p_x$ are isomorphisms. If there is a map $z \ra x$, one can  see that $\theta_{x(i-1)}$ and $\theta_{xi}$ are both isomorphisms since $x \notin \utmii$ and $x \notin \utmi$. This implies that $\lambda_x$ and $p_x$ are also isomorphisms. 
	\end{enumerate}
	We thus obtain the desired result. 
\end{proof}

We are now ready to prove the main result of this section: Proposition~\ref{theta_prop}. 

\begin{proof}[Proof of Proposition~\ref{theta_prop}]
	We need to show that $\Theta (f) = F$ and $\Theta (f') = F'$ are weakly equivalent in $\fuca$. By Proposition~\ref{fuca_prop}, it is enough to show that $\phi F \simeq \phi F'$ in $\mcalf_A(\utm; \mcalm)$ (see Definition~\ref{fuca_defn}). Here $\phi \colon \fuca \lra \mcalf_A(\utm; \mcalm)$ is the obvious functor defined by $\phi(G) = G$. From now on, we will regard $F$ and $F'$ as objects of $\mcalf_A(\utm; \mcalm)$. As before, we let $F_i$ (respectively $F'_i$) denote the restriction of $F$ (respectively $F'$) to $\utmi$. Taking the adjoint to (\ref{fi_diag}),  we get (\ref{fi2_diag}), which is of course a commutative diagram.  
	\begin{eqnarray} \label{fi2_diag}
	\xymatrix{F_i \ar[d]  & &  \fib \ar[ll]_-{\sim}^-{\alpha_i} \ar[rr]^-{\sim}_-{\alpha'_i} \ar[d]  & &  F'_i \ar[d] \\
		E_{ij} F_j   & & E_{ij} \fjb \ar[ll]_-{\sim}^-{E_{ij} (\alpha_j)} \ar[rr]^-{\sim}_-{E_{ij} (\alpha'_j)} & & E_{ij} F'_j.}
	\end{eqnarray}
	From Lemma~\ref{facts_lem}~-(ii), it is easy to see why the maps $E_{ij}(\alpha_j)$ and $E_{ij}(\alpha'_j)$ are both weak equivalences. Applying now the functor $E_i$ to (\ref{fi2_diag}), and using Lemma~\ref{facts_lem}~-(i), we get 
	\begin{eqnarray} \label{fi3_diag}
	\xymatrix{E_iF_i \ar[d]  &  E_i\fib \ar[l]_-{\sim}^-{\beta_i} \ar[r]^-{\sim}_-{\beta'_i} \ar[d] &  E_iF'_i \ar[d] \\
		E_j F_j    & E_j\fjb \ar[l]_-{\sim}^-{\beta_j} \ar[r]^-{\sim}_-{\beta'_j}  & E_j F'_j,}
	\end{eqnarray}
	where $\beta_i:= E_i(\alpha_i)$. This gives rise to  two weak equivalences: $\xymatrix{\ebb & \overline{\ebb} \ar[l]_-{\sim}^-{\beta}  \ar[r]^-{\sim}_-{\beta'} &  \ebb'}$, where  $\overline{\ebb} \colon \nbb \lra \mcalm^{\utm}$ is the obvious functor defined by $\overline{\ebb}(i) = E_i \fib$. Recalling the map $\kappa$ from Lemma~\ref{kappa_lem}, we have the  zigzag 
	\begin{eqnarray} \label{fi4_diag}
	\xymatrix{F \ar[r]^-{\sim}_-{\kappa} & \textup{holim}_{\nbb} \ \ebb   & & \textup{holim}_{\nbb} \ \overline{\ebb} \ar[ll]_-{\sim}^-{\text{holim}(\beta)}  \ar[rr]^-{\sim}_-{\text{holim}(\beta')} & & \textup{holim}_{\nbb} \ \ebb' &  F', \ar[l]_-{\sim}^-{\kappa'}}
	\end{eqnarray}
	which completes the proof \footnote{One may ask the question to know whether the objects that appear in diagrams (\ref{fi2_diag}), (\ref{fi3_diag}), and (\ref{fi4_diag}) belong to $\mcalf_A(\utm; \mcalm)$. The answer is yes. This is straightforward by using Lemma~\ref{facts_lem}~-(ii) and some classical properties of  homotopy limit.}. 
\end{proof}

We close this section with the following result.

\begin{prop} \label{fuca_iso_prop}
	The map $\lamb$ from Definition~\ref{lamb_defn} is a bijection. That is,  
	\[
	\xymatrix{\fuca/\textgoth{w} \ar[r]_-{\cong}^-{\lamb} & [\tmd, \ahd]}. 
	\]
	In fact the inverse of $\lamb$ is the map $\ttb$ from Definition~\ref{ttb_defn}. 
\end{prop}

\begin{proof}
	From the definitions, it is easy to see that $\Theta\Lambda (F) \simeq F$ for every $F \in \fuca$. So $\ttb \circ \lamb = id$. On the other  hand, let $f \colon \tmd \lra \ahd$ be a simplicial map.  By construction $\Theta(f)$ is fibrant since $\fs \colon \dtn \lra \ca$ is fibrant for any $n \geq 0, \sigma \in \tm_n$. Therefore $\mcalrb \Theta (f) = \Theta (f)$, where $\mcalrb$ is the fibrant replacement funtor from (\ref{mcalrb_eqn}). This implies that  $\Lambda \Theta (f) = f$, and therefore $\lamb \circ \ttb = id$, which completes the proof.  
\end{proof}

\section{Proof of the main result}   \label{pmr_section}

We now have all ingredients to prove Theorem~\ref{main_thm}, which is the main result of this paper.  Roughly speaking it classifies a class of functors called \textit{homogeneous} that we now define. 

Some preliminaries definitions are needed. Let $F \colon \om \lra \mcalm$ be a functor. 
 \begin{itemize}
 	\item The functor $F$ is called {\it good} if it is an isotopy functor (see Definition~\ref{defn:we_categories}) and for any string $U_0 \to U_1 \to \cdots$ of morphisms of $\om$, the natural map $F(\cup_{i=0}^\infty U_i) \to \underset{i}{\textup{holim}} \, F(U_i)$ is a weak equivalence.
 	
 	\item The functor $F$ is {\it polynomial of degree $\leq k$} if for every $U \in \om$ and pairwise disjoint closed subsets $A_0, \cdots, A_k$ of $U$, the canonical map 
 	$
 	F(U) \lra \underset{S \neq \emptyset}{\text{holim}} \; F(U \backslash \cup_{i \in S} A_i)
 	$
 	is a weak equivalence. Here $S \neq \emptyset$ runs over the power set of $\{0, \cdots, k\}$.
 	
 	\item The {\it $k$th polynomial approximation} to $F$, denoted $T_kF$, is the functor $T_k F \colon \om \lra \mcalm$ defined as 
 	$
 	T_k F (U) =  \underset{V \in \ok (U)}{\text{holim}} F(V).
 	$
 	On morphisms, $T_kF$ is defined in the obvious way.
 \end{itemize} 

 \begin{defn} \label{defn:homogeneous_functor}
 	Assume that the simplicial model category $\mcalm$ has a terminal object denoted $0$. A functor $F \colon \om \lra \mcalm$ is said to be \emph{homogeneous of degree} $k$ if it satisfies the following three conditions:
 	\begin{enumerate}
 		\item $F$ is a good functor;
 		\item $F$ is polynomial of degree $\leq k$; 
 		\item The unique map $T_{k-1} F (U) \lra 0$ is a weak equivalence for every $U \in \om$. 
 	\end{enumerate}   
 \end{defn}

\begin{defn} \label{defn:fkom}	 
		 Define $\mcalf_{kA}(\om; \mcalm)$ as the category of homogeneous functors $F \colon \om \lra \mcalm$ of degree $k$  such that  $F(U) \simeq A$ for every $U$ diffeomorphic to the disjoint union of exactly $k$ open balls. 
\end{defn}

\begin{lem} \label{fkfu_lem}
	Let $\mcalm$ be a simplicial model category that has a zero object, and let $\fkam$ be as in Definition~\ref{defn:fkom}. Let $\mcalt^{F_k(M)}$ be a triangulation of $F_k(M)$, and let $\mcalf_A(\utfm; \mcalm)$ be the category from Definition~\ref{fum_defn}.  Then the categories $\fkam$ and $\mcalf_A(\utfm; \mcalm)$ are weakly equivalent. 
\end{lem}

\begin{proof}
	By using the same approach as that we used to prove \cite[Theorem 1.3 \& Lemma 6.5]{paul_don17-2}, one has 
	\begin{eqnarray} \label{we_eq0}
	\fkam \simeq \foam,
	\end{eqnarray}
	and
	\begin{eqnarray} \label{we_eq1}
	\foam \simeq \mcalf_A (\mcalb_{\utfm}; \mcalm),
	\end{eqnarray}
	where $\mcalf_A (\mcalb_{\utfm}; \mcalm)$ is the category from Definition~\ref{fum_defn}. Furthermore, by Proposition~\ref{fum_butm_prop}, we have
	\begin{eqnarray} \label{we_eq2}
	\mcalf_A (\mcalb_{\utfm}; \mcalm) \simeq \mcalf_A (\utfm; \mcalm). 
	\end{eqnarray}
	Combining (\ref{we_eq0}), (\ref{we_eq1}), and (\ref{we_eq2}), we get the desired result.  
\end{proof}

We are now ready to prove  Theorem~\ref{main_thm}.

\begin{proof}[Proof of Theorem~\ref{main_thm}]
	For the first part, we refer the reader to the introduction. The second part is proved as follows. Let $\mcalt^{F_k(M)}$ as above. From Lemma~\ref{fkfu_lem} and \cite[Remark 6.4]{paul_don17-2}, we deduce that the localizations of $\fkam$ and $\mcalf_A (\utfm; \mcalm)$  are equivalent in the classical sense. Furthermore,  by Proposition~\ref{fuca_prop}, we have that the localization of the latter category is equivalent to the localization of $\mcalf (\utfm; \ca)$. This implies that weak equivalence classes of $\fkam$ are in one-to-one correspondence with weak equivalence classes of $\mcalf (\utfm; \ca)$. That is,
	\begin{eqnarray} \label{pthm_eqn}
	\fkam \slash \textgoth{w} \ \cong \ \mcalf (\utfm; \ca) \slash \textgoth{w}. 
	\end{eqnarray} 
	Applying Proposition~\ref{fuca_iso_prop}, we get 
	\begin{eqnarray} \label{pthm_eqn2}
	\mcalf (\utfm; \ca) \slash \textgoth{w}  \   \cong  \  \left[\mcalt^{F_k(M)}_{\bullet}, \ahd\right]. 
	\end{eqnarray}
	Define $\ah$ to be the geometric realization of $\ahd$. That is, $\ah:= |\ahd|$. Since $|\mcalt^{F_k(M)}_{\bullet}| \cong F_k(M)$, it follows that 
	$
	\left[\mcalt^{F_k(M)}_{\bullet}, \ahd\right] \cong \left[F_k(M), \ah\right]. 
	$ This proves the theorem. 
\end{proof}

\section{How our classification is related to that of Weiss}  \label{comparison_section}

In this section we briefly  recall the Weiss classification of homogeneous functors, and we explain a connection to our classification result (Theorem~\ref{main_thm}). As usual, we let \text{Top} (respectively $\text{Top}_*$) denote the category of spaces (respectively pointed spaces). 


We begin with Weiss' result about the classification of homogeneous functors. Let $p \colon Z \lra F_k(M)$ be a fibration. Define $F \colon \om \lra \text{Top}$ as $F(U) = \Gamma(p; F_k(U))$, the space of sections of $p$ over $F_k(U)$. It turns out that $F$ is polynomial of degree $\leq k$ (see \cite[Example 7.1]{wei99}). Define another functor $G \colon \om \lra \text{Top}$ as follows. Let $M^k \slash \Sigma_k$ denote the orbit space of the action of the symmetric group $\Sigma_k$ on the $k$-fold product $M^k$. Let $\Delta_k M$ be the complement of $F_k(M)$ in $M^k \slash \Sigma_k$. (The space $\Delta_k M$ is the so-called \textit{fat diagonal} of $M$.) Define
$
G(U) := \underset{N \in \mcaln}{\text{hocolim }} \Gamma(p; F_k(U) \cap N),
$
where  $\mcaln$ stands for the poset of neighborhoods of $\Delta_k M$. The space $G(M)$ should be thought of as the space of sections near the fat diagonal of $M$. 
Let $\eta \colon F \lra G$ be the canonical map induced by the inclusions $F_k(U) \cap N \subseteq F_k(U)$. It turns out that $\eta$ is  nothing but the canonical map $T_k F \lra T_{k-1}F$ since $T_{k-1}F$ is equivalent to $G$  (see \cite[Propositions 7.5 and 7.6 ]{wei99}). Selecting a point $s$ in $G(M)$, if one exists, we define $E_{p,s} \colon \om \lra \textup{Top}$  as the homotopy fiber of $\eta$ over $s$, that is, $E_{p, s} := \textup{hofiber} (\eta)$. It follows from \cite[Example 8.2]{wei99} that $E_{p, s}$ is homogeneous of degree $k$.
Weiss' classification says that every homogeneous functor can be constructed in this way. Specifically, we have the following result. 

\begin{thm} \cite[Theorem 8.5]{wei99}  \label{thm:wchf}
	Let $E \colon \om \lra \textup{Top}$ be homogeneous of degree $k$. Then there is a (levelwise) homotopy equivalence $E \lra E_{p, s}$ for a fibration $p \colon Z \to F_k(M)$ with a section $s$ near the fat diagonal of $M$. 
\end{thm}

Weiss also describes the fiber of $p$ in terms of $E$. 

\begin{prop}  \cite[Proposition 8.4 and Theorem 8.5]{wei99} \label{prop:fiber}
	Let $E$ be as in Theorem~\ref{thm:wchf}, and suppose $E$ is classified by a fibration $p \colon Z \to F_k(M)$. Then the fiber $p^{-1}(S)$ over $S \in F_k(M)$ is homotpy equivalent to $E(U_S)$, where $U_S$ is a tubular neighborhood of $S$, so that $U_S$ is diffeomorphic to the dijoint union of $k$ open balls. 
\end{prop}	

Combining Theorem~\ref{thm:wchf} and Proposition~\ref{prop:fiber}, we get that the classification of the objects of $\mcalf_{kA}(\om; \text{Top})$ (see Definition~\ref{defn:fkom})  amounts to the classification of fibrations over $F_k(M)$ with a section $s$ near the fat diagonal, and whose fiber is homotopy equivalent to $A$. Note that in the case $k=1$ the fat diagonal is empty, and we just look at fibrations over $M$ with fiber $A$.

We now explain a connection to our classification result. 
As usual,  $\mcalm$ is an arbitrary simplicial model category, and $A$ is an object of $\mcalm$. Assume that $A$ is both fibrant and cofibrant. Let $\haut A$ be the simplicial monoid of self weak equivalences $A \stackrel{\sim}{\to} A$. Define B$\haut A:= \overline{W}\haut A$, where $\overline{W}$ is the functor from \cite[p. 87]{may92} or  \cite[p. 269]{goe_jar09}. (Note that $\overline{W}$ lands in the category of simplicial sets.) We still denote by B$\haut A$ the geometric realization of B$\haut A$. So depending on the context, $\text{Bhaut}A$ should be interpreted as either a simplicial set or a topological space. We have the following two conjectures.  
\begin{conj} \label{conj:ah_BhautA}
	Let $\mcalm$ be a simplicial model category, and let $A \in \mcalm$ be an object which is both fibrant and cofibrant. Let $\ahd$ be the simplicial set from Section~\ref{ahd_subsection} constructed out of $A$, and let $\ah$ be its geometric realization. Then $\ah$ is weakly equivalent to $\emph{Bhaut} A$. That is, $\ah \simeq \emph{Bhaut} A$. 
\end{conj}

\begin{conj} \label{conj:thm}
	Let $\mcalm$ and $A$ be as in Conjecture~\ref{conj:ah_BhautA}. Then for any manifold $M$,
	\begin{enumerate}
		\item[(i)] if $k = 1$, there is a bijection $\mcalf_{1A}(\om; \mcalm)\slash \textgoth{w} \ \cong \big[M, \emph{Bhaut} A\big]$,
		\item[(ii)] if $k \geq 2$ and $\mcalm$ has a zero object, there is a bijection $\fkam \slash \textgoth{w} \ \cong \big[F_k(M), \emph{Bhaut} A\big]$. 
	\end{enumerate}
\end{conj} 

One way to get the latter is to prove Conjecture~\ref{conj:ah_BhautA} and then use our Theorem~\ref{main_thm}. When $\mcalm$ is the category of spaces, we have that B$\haut A$ classifies fibrations with fiber homotopy equivalent to $A$ (see \cite[Corollary 9.5]{may75}). 

We wanted to prove Conjecture~\ref{conj:thm} and state it as the main result of this paper, but we couldn't find a way to do it.  We tried another interesting approach, which does not involve $\ah$ at all, that we now explain.

\subsection{Another attempt to demonstrate Conjecture~\ref{conj:thm}}

We discuss the case $k=1$; the cases $k \geq 2$ can be treated similarly.  In the introduction we provided the proof of the first part of Theorem~\ref{main_thm} that can be summarized as
\[
\mcalf_{1A}(\om; \mcalm)\slash \textgoth{w} \ \cong \ \mcalf(\utm; \ca)\slash \textgoth{w} \ \cong \big[ M, \ah\big]. 
\]
For the purposes of the new approach, we need to replace $\ca$ with a larger category $\ca'$ defined as follows. The objects of $\ca'$ are the objects of $\mcalm$ which are  connected to $A$ by a zigzag of weak equivalences. The morphisms of $\ca'$ are the weak equivalences between its objects.  By definition $\ca'$ is connected. Note that the category $\ca'$ is nothing but what Dwyer and Kan  call \textit{special classification complex} of $A$ (see \cite[Section 2.2]{dwyer_kan84} where they use the notation $scA$ instead).  The usefulness of $\ca'$ is due to the fact that its (enriched) nerve (which is denoted below by $d\Nt \ca'$) has the homotopy type of $\text{Bhaut}A$ \cite[Proposition 2.3]{dwyer_kan84}, and depends only on the weak equivalence class of $A$.

Using the same approach as that we used to prove Theorem~\ref{main_thm}, one can show that there is a bijection
\[
\mcalf_{1A}(\om; \mcalm)\slash \textgoth{w} \ \cong \ \mcalf(\utm; \ca')\slash \textgoth{w}.
\]
This makes sense if and only if $\mcalf_{1A}(\om; \mcalm)$ is viewed as the category of  homogeneous functors of degree $1$ whose morphisms are natural transformations which are (objectwise) weak equivalences.  The next step is to try to see whether $\mcalf(\utm; \ca')\slash \textgoth{w}$ is in bijection with $[M, \text{Bhaut}A]$. For this, consider the diagram

\[
\xymatrix{\mcalf(\utm; \ca')\slash \textgoth{w} \ar[r]^-{\widetilde{N}}_-{\cong}  &  \underset{\text{s}^2\text{Set}}{\text{Hom}} (\widetilde{N}\utm, \widetilde{N}\ca')/\simeq  &  \underset{\text{sSet}}{\text{Hom}} (N\utm, N\ca')/\simeq \ar[l]_-{\alpha}^-{\cong} \ar[d]^-{\psi_*} \\
	\big[ |N\utm|, |\mcalr d\widetilde{N}\ca'|\big] &  \underset{\text{sSet}}{\text{Hom}} (N\utm, \mcalr d\widetilde{N}\ca')/\sim \ar[l]_-{|-|}^-{\cong}  &  \underset{\text{sSet}}{\text{Hom}} (d\widetilde{N}\utm, d\widetilde{N}\ca')/\simeq \ar[l]_-{\varphi_*}  \\
	\big[ M, |d\widetilde{N}\ca'|\big] \ar[u]^-{|\varphi|_*}_-{\cong} \ \ \  \  \  \  \ \   \cong & \! \! \!  \! \! \! \! \! \! \!  \! \! \! \! \! \! \!  \! \! \! \! \! \! \!  \! \! \! \! \! \! \!  \! \! \! \! \! \! \!  \! \! \! \! \!  \big[ M, \text{Bhaut} A\big], }
\]

where
\begin{enumerate}
	\item[$\bullet$] $\utm$ is viewed as a simplicially enriched category with the constant simplicial set at $\underset{\utm}{\text{Hom}} (u, v)$ for every $u, v \in \utm$. The simplicial enrichment of $\ca' \subseteq \mcalm$ is of course induced by that of $\mcalm$. Note that since the source, $\utm$, is discrete,  we have that every functor  $F \colon \utm \to \ca'$ is a simplicial functor in the sense that $F$ respects the simplicial enrichment;
	\item[$\bullet$] $\text{s}^2\text{Set}$ is the category of bisimplicial sets, and $\widetilde{N}$ is the bisimplicial nerve functor defined as follows. For a simplicially enriched category $\mcala$, $\widetilde{N} \mcala$ is the simplicial object in simplicial sets whose $k$-simplices are given by 
	\[
	(\widetilde{N} \mcala)_k = \coprod_{(a_0, \cdots, a_k) \in \mcala^{k+1}} \prod_{i=0}^{k-1} \underset{\mcala}{\textbf{Hom}} (a_i, a_{i+1}),
	\]
	where $\underset{\mcala}{\textbf{Hom}} (-, -)$ stands for the simplicial hom-set functor. It turns out that $\Nt$ is fully faithful.  The equivalence relation ``$\simeq$'' that appears in $\underset{\text{s}^2\text{Set}}{\text{Hom}} (\widetilde{N}\utm, \widetilde{N}\ca')/\simeq$ is the one generated by homotopies. That is, two bisimplicial maps $f, g \colon \Nt \utm \to \Nt \ca'$ are in relation with respect to $\simeq$ if they are connected by a zigzag $f \leftarrow f_1 \rightarrow \cdots \leftarrow f_n \rightarrow g$ where $f_i \to f_j$ means that there is a homotopy from $f_i \to f_j$;
	\item[$\bullet$] $N$ is the ordinary/discrete nerve functor, sSet is the category of simplicial sets as usual, and the equivalence relation ``$\simeq$'' that appears in $\underset{\text{sSet}}{\text{Hom}}(N\utm, N\ca')/\simeq$ is generated by homotopies in the same sense as before. The isomorphism $\alpha$ is defined in the standard way. 
	\item[$\bullet$] $d$ is the diagonal functor $d \colon \text{s}^2\text{Set} \to \text{sSet}$, and $\psi_*$ is the map induced by the canonical map $\psi \colon N\ca' \to d\Nt \ca'$. Note that by definition $d\Nt\utm = N\utm$. 
	\item[$\bullet$] $\mcalr$ is a fibrant replacement functor $\mcalr \colon \text{sSet} \to \text{sSet}$, $\sim$ is the usual homotopy relation, and $\varphi_*$ is the map induced by the fibrant replacement $\varphi: \xymatrix{d\Nt\ca' \  \ar@{>->}[r]^-{\sim} & \mcalr d\Nt\ca'}$; 
	\item[$\bullet$] $|-|$ is of course the geometric realization functor, and $|\varphi|_*$ is the map induced by $|\varphi|$. Note that $|N\utm|$ is homeomorphic to $M$;
	\item[$\bullet$] The final bijection comes from  the fact that $|d\Nt\ca'| \simeq \text{Bhaut}A$  as mentioned above. 
\end{enumerate} 

This is another  candidate  for a possible proof of Conjecture~\ref{conj:thm}, but the problem here is that we do not know how to show that  $\psi_*$ and $\varphi_*$ are both isomorphisms.

\cleardoublepage

\textsf{University of British Columbia -- Okanagan, 3333 University Way, Kelowna, BC V1V 1V7, Canada\\
Department of Computer Science, Mathematics, Physics and Statistics\\}
\textit{E-mail address: paul.tsopmene@ubc.ca}

\textsf{University of Regina, 3737 Wascana Pkwy, Regina, SK S4S 0A2, Canada\\
	Department of Mathematics and Statistics\\}
\textit{E-mail address: donald.stanley@uregina.ca}

\end{document}